\documentclass[a4paper,11pt]{amsart}
\usepackage{amsmath,amssymb,amscd,amsfonts}
\usepackage{amsrefs, esint}
\numberwithin{equation}{section}
\newtheorem{theorem}{Theorem}[section]

\newtheorem{rem}[theorem]{Remark}
\newtheorem{defn}[theorem]{Definition}
\newtheorem{lemma}[theorem]{Lemma}

\newtheorem{prop}[theorem]{Proposition}

\newtheorem{corollary}[theorem]{Corollary}

\def\det{\mathop{\rm det}\nolimits}

\def\Re{\mathop{\rm Re}\nolimits} 
\def\Im{\mathop{\rm Im}\nolimits} 

\def\tr{\mathop{\rm tr}\nolimits}

\def\dbar{\bar\partial}
\def\ddbar{\partial\bar\partial}
\def\d{\partial}

\def\cG{{\mathcal G}}

\def\cR{{\mathcal R}}
\def\cE{{\mathcal E}}

\def\cC{{\mathcal C}}
\def\cM{{\mathcal M}}
\def\cK{{\mathcal K}}
\def\cV{{\mathcal V}}

\def\cU{{\mathcal U}}

\def\cH{{\mathcal H}}

\let\ep=\varepsilon
\let\vp=\varphi 

\def\bC{{\mathbb C}}

\def\a{{\alpha}}

\def\g{\gamma}

\def\b{\beta}

\def\k{{\kappa}}

\def\l{{\ell}}

\def\t{{\tau}}
\def\kf{{\mathfrak{f}}}

\title[ geodesics approximation]
{Approximation of weak geodesics and subharmonicity of Mabuchi energy, II: $\ep$-geodesics}

\author{Long Li}
\address{Mathematics Institute of ShanghaiTech University, 393 Middle Huaxia Road, Pudong 201210, Shanghai, China}
\email{lilong1@shanghaitech.edu.cn}

\address{Science Institute, University of Iceland, Reykjavik, Iceland.}
\email{ longli@hi.is }

\begin{document}
\maketitle 

\begin{abstract}
The purpose of this article is to study the strict convexity 
of the Mabuchi functional along a $\cC^{1,\bar 1}$-geodesic,
with the aid of the $\ep$-geodesics. We proved the $L^2$-convergence 
of the fiberwise volume element of the $\ep$-geodesic. 
Moreover, the geodesic is proved to be uniformly fiberwise non-degenerate
if the Mabuchi functional is $\ep$-affine.

\end{abstract}

\section{Introduction}
In order to study the uniqueness and existence problems of the K\"ahler-Einstein metrics 
on a Fano manifold $X$, Mabuchi (\cite{Ma}, \cite{BM}) introduced a useful energy functional $\cM$,
called the $K$-energy or Mabuchi functional, on the space $\cH_{[\omega]}$.
This is the space of all smooth K\"ahler potentials 
in a given cohomology class $[\omega]$ on $X$,
and it can be written as 
$$ \cH_{[\omega]}: = \{ \vp\in C^{\infty}(X); \ \ \omega+ dd^c\vp >0 \}.   $$
Through out this paper, we always normalize the K\"ahler class as 
$\int_X \omega^n =1$. 

It is also observed by Mabuchi (\cite{Ma}) that this functional $\cM$ is convex along a smooth geodesic $\cG$, 
connecting with arbitrary two points in $\cH_{[\omega]}$. 
Moreover, if the Mabuchi functional is affine along a smooth geodesic $\cG$,
then $\cG$ must be generated by a holomorphic vector field. 
In this case, we say that the Mabuchi functional is strictly convex along the geodesic $\cG$. 

It turns out that the Mabuchi functional $\cM$ has also played an important role in the study of the 
constant scalar curvature K\"ahler(cscK) metrics. 
First, the convexity of $\cM$ is crucial in the proof of the uniqueness of the cscK metrics (\cite{BB}, \cite{CLP}, \cite{CPZ}).
Second,  on a K\"ahler manifold, the asymptotic behavior of this convex function along a geodesic ray $\cG$ is an invariant $\Upsilon$ (\cite{CC1}, \cite{CC2}). 
It is proved  that this manifold admits a cscK metric if and only if it is geodesic-stable (Theorem (1.2), \cite{CC2}).
That is to say, either we have $\Upsilon>0$, or $\Upsilon =0$ and the ray $\cG$ is parallel to another geodesic ray generated by a holomorphic vector field.

However, one can not expect that there always exists a smooth geodesic connecting 
two points in $\cH_{[\omega]}$, due to the example in Darvas-Lempert (\cite{LD}).  
In realty, we can only rely on  the so called $\cC^{1, \bar 1}$-geodesic $\cG$. 
As the solution of the homogenous complex Monge-Amp\`ere equation (see equation (\ref{pre-0000})),
it always exists (\cite{C00}) between arbitrary two points in $\cH_{[\omega]}$.
However, the difficulties to deal with a $\cC^{1, \bar 1}$-geodesic 
are the lack of the regularities and the possible degeneracy of $\cG$. 

Nevertheless, the Mabuchi functional $\cM$ was proved to be 
convex and continuous along a $\cC^{1,\bar 1}$-geodesic, 
by the work of Berman-Berndtsson (\cite{BB}) and also Chen-Li-P\u aun (\cite{CLP}). 
Then people try to ask the question if the Mabuchi functional is also strictly convex 
along such a geodesic. 

If the boundary of a geodesic segment  $\cG$ has merely $\cC^{1,\bar 1}$-regularities, 
then the answer is negative due to the example of Berman (\cite{Ber}).
However, the situation becomes very different if we require a stronger condition
on the boundary regularities of $\cG$. 
The answer is affirmative  (\cite{LL}), when 
$\cG$ is connecting with two non-degenerate energy minimizers of $\cM$. 
According to the work of He-Zeng (\cite{HZ}),
the boundary of $\cG$ is actually smooth in this case.  
Another example is on a toric K\"ahler manifold. 
Then it is well known that a geodesic segment must be smooth and non-degenerate 
if it is boundary is.
In these two examples, 
the strict convexity of $\cM$ follows from the smoothness of the boundary of the geodesic. 
For this reason, we always assume that the boundary of a geodesic $\cG$
belongs to the space $\cH_{[\omega]}$ through out this paper. 

In order to circumvent the difficulties arising from a $\cC^{1,\bar 1}$-geodesic $\cG$,
we are appealing to using the so called $\ep$-geodesic $\cG_{\ep}$ (\cite{C00}). 
The $\ep$-geodesic $\cG_{\ep}$  is a sequence of smooth  approximation of $\cG$,
satisfying a non-homogenous complex Monge-Amp\`ere equation (see equation (\ref{int-001})). 
In fact, we have proved in (\cite{CLP}) that the Mabuchi functional is \emph{almost convex}
along the $\ep$-geodesic. 

However, the difficulty is that the convergence $\cG_{\ep} \rightarrow \cG$
is merely weakly $L^p$ for all $p >1$. 
In other words, if we denote 
$\vp_{\ep}(t, \cdot)$ as the fiberwise $\ep$-geodesic potential of $\cG_{\ep}|_{X_t}$, and $\vp(t,\cdot )$ the fiberwise geodesic potential of $\cG|_{X_t}$,
then it is not clear to us whether the 
energy $\cM(\vp_{\ep})$
converges to the energy $\cM(\vp)$ or not. 
Therefore, we can not directly conclude the convexity of $\cM$ 
by using the $\ep$-geodesic in \cite{CLP}.

The new observation is that the convergence $\cM(\vp_\ep) \rightarrow \cM(\vp)$
is true, provided that the Mabuchi function $\cM$ is affine along $\cG$ (see Theorem (\ref{thm-ae-001}) and  Corollary (\ref{cor-l2-001})). 
This first leads us to the following $L^2$-convergence of the fiberwise volume element of $\cG_{\ep}$. 
Write $\omega_{\ep}: = \cG_{\ep}|_{X_t}$ and $\omega_{\vp}: = \cG|_{X_t} $ on a fiber $X_t: = \{t\}\times X $
for any $t\in [0,1]$.

\begin{theorem}[Theorem (\ref{thm-l2-001})]
\label{thm-intr-001}
Suppose the Mabuchi functional $\cM$ is affine along a $\cC^{1,\bar 1}$-geodesic $\cG$. 
Then the fiber-wise volume element of the $\ep$-geodesic converges to the 
volume element of the geodesic in the strong $L^2$ sense.
In other words, we have on each fiber $X_t$
\begin{equation}
\label{intr-000}
 \frac{\omega_{\ep}^n}{\omega^n} \rightarrow \frac{\omega_{\vp}^n}{\omega^n}, \ \ \  \ep\rightarrow 0,  
\end{equation}
under the $L^2$-norm, possibly after passing to a subsequence. 
\end{theorem}

We emphasis that the $L^2$-convergence of the volume element (equation (\ref{intr-000})) may not be true in general. 
Next, a slightly stronger condition than the affine Mabuchi functional will be introduced,  
aiming to resolve the possible degeneracy on the geodesic $\cG$. 
Write the restriction of the Mabuchi functional along a $\ep$-geodesic $\cG_{\ep}$ as 
$$\cK_{\ep} (t) : = \cM (\vp_{\ep} (t))$$ 
for all $t\in[0,1]$. 
Then we say that the Mabuchi functional is \emph{$\ep$-affine} along the geodesic $\cG$ if it satisfies 
\begin{equation}
\label{intr-001}
 \frac{d \cK_{\ep}}{dt}|_{t=1}  - \frac{d \cK_{\ep}}{dt}|_{t=0} = O(\ep),
\end{equation}
for all $\ep>0$ small (see Definition (\ref{def-ch-002})). 
We note that $\cG$ and $\cG_{\ep}$ are uniquely determined if the boundary of $\cG$ is given. 
In fact, 
the Mabuchi functional must be  
affine along the $\cC^{1,\bar 1}$-geodesic $\cG$ if it is $\ep$-affine along each $\cG_{\ep}$ (see Lemma (\ref{lem-eaf-0010})).
Moreover, we proved the following result.

\begin{theorem}[Theorem \ref{thm-ch-001}]
\label{thm-intr-002}
Suppose the Mabuchi functional $\cM$ is $\ep$-affine along a $\cC^{1,\bar 1}$-geodesic $\cG$.
Then $\cG$ is uniformly fiberwise non-degenerate, 
namely, there exists a uniform constant $\k_1 >0$ such that 
$$ \cG|_{X_t} > \kappa_1 \omega,$$
for almost everywhere $t\in [0,1]$.
\end{theorem}

There are three main \text{steps} of the proof for the above Theorem. 
\textbf{Step (1)} is to figure out the so called \emph{gap phenomenon} of the geodesic $\cG$, 
which is first observed in our previous work (\cite{LL}). 
\textbf{Step (2)} is to establish a kind of $W^{1,2}$-estimate for the volume element $\omega_{\ep}^n$,
which is provided from the $\ep$-affine condition. 
\textbf{Step (3)} (see Proposition (\ref{prop-nd-001})) is to prove 
that a non-negative function must have a positive lower bound 
if it has the gap phenomenon and satisfies a certain partial $W^{1,2}$-estimate.

As an application of our Theorem (\ref{thm-intr-001}) and (\ref{thm-intr-002}), 
we can further estimate (see Theorem (\ref{thm-app-001})) the limit of the complex hessian  
of $\cM(\vp_{\ep})$, if the Mabuchi functional $\cM(\vp)$ is $\ep$-affine. 
More precisely, we have 
\begin{equation}
\label{intr-003}
\limsup_{\ep\rightarrow 0} \cK''_{\ep}(t)  \geq 0,
\end{equation}
for almost everywhere $t\in [0,1]$. 

Another application is that we can recover the strict convexity result 
in (\cite{LL}), as stated before.   
Moreover, if the manifold $X$ satisfies $c_1(X) = 0$
or $c_1(X) <0$, then we can utilize Chen's argument (\cite{C00}) to conclude the strict convexity of the Mabuchi functional,
provided the $\ep$-affine condition (see Theorem (\ref{thm-app-003})). 

Therefore, we conjecture that the Mabuchi functional is strictly convex
along a $\cC^{1,\bar 1}$-geodesic, if it is $\ep$-affine.
In fact, the convergence (equation (\ref{intr-003})) further implies 
an $L^2$-estimate on $\dbar v_{\ep}$ for a sequence of smooth vector fields $v_{\ep}$ (see Theorem (\ref{thm-app-001})). 
Unfortunately, there is still some difficulties to conclude the holomorphicity of $v_{\infty}$ as the limit of this sequence $v_{\ep}$. 

Finally, it might be worthy to point out that 
a geodesic $\cG$ possibly possess more regularities than $\cC^{1,\bar 1}$,
when the Mabuchi functional is $\ep$-affine along it. 
Hopefully, we will see more examples about this fact, and 
the regularity problem will be considered in our following projects. 

\bigskip

\textbf{Acknowledgment: }
The author is very grateful to Prof. Chen and Prof. P\u aun  who introduced this problem,
and have given continuous encouragement. 
He also wants to thank Prof. Chengjian Yao, Dr. Jingchen Hu and Prof. Wei Sun for lots of useful discussion. 
Finally, he thanks the referee who gave many valuable suggestions to improve this paper.

\bigskip

\section{Preliminary}

Denote $\Gamma$ by the following strip domain in $\bC$
$$ \Gamma: = \{  \tau\in \bC; \ \ 0\leq \Re \tau \leq 1     \}, $$
and we write the complex coordinate $\t$ as  $t + \sqrt{-1}s$. 
Consider the product manifold $Y: = \Gamma\times X$, 
and then $Y$ is a complex K\"ahler manifold with boundary. 
Assume that $\pi$
is the holomorphic projection from the product space $Y$ to $X$. 
Let $\Phi$ be a quasi-plurisubharmonic function on $Y$ continuous up to the boundary. 
Denote $\cG$ by the following closed positive $(1,1)$ current 
\begin{equation}
\label{pre-0}
 \pi^*\omega + dd^c\Phi \geq 0
\end{equation}
on $Y$. 
We say that $\cG$ is a geodesic in the space of K\"ahler potential,
if the function $\Phi$ is independent of $s = \Im \tau$, and satisfies the following
 \emph{Homogeneous complex Monge-Amp\`ere} (HCMA) equation on $Y$
\begin{equation}
\label{pre-0000}
\cG^{n+1} = ( \pi^*\omega + dd^c\Phi )^{n+1} = 0.
\end{equation}
This wedge product performs in the sense of Bedford and Talyor (\cite{BT}).

The boundary value of $\Phi$ is required to be in the space $\cH_{[\omega]}$,
and we say that $\cG$ is a geodesic connecting two points $\varphi_0, \varphi_1\in \cH_{[\omega]}$
if 
$$\Phi|_{X\times \{0 \}} = \varphi_0; \ \ \ \Phi|_{X\times\{1\}} = \varphi_1.$$ 

It is proved by Chen (\cite{C00}) that such a geodesic always exists,
and is unique with fixed boundary value.
It also has the so called $\cC^{1,\bar 1}$-regularities on $Y$,
namely, writing $\cG$ locally as  
$$ g_{\t\bar\t} d\t\wedge d\bar\t + \sum_{\a,\b =1}^n ( g_{\t\bar\b} d\t\wedge d\bar z^{\b} + g_{\a\bar\t}dz^{\a}\wedge d\bar\t + g_{\a\bar\b} dz^{\a}\wedge d\bar z^{\b} ),$$
we have 
$$ || g_{\t\bar\t} ||_{L^{\infty}} +  \sum_{\a,\b =1}^n  (   ||g_{\t\bar\b}||_{L^{\infty}} +    ||g_{\a\bar\t}||_{L^{\infty}} +  ||g_{\a\bar\b}||_{L^{\infty}}    )  < +\infty. $$
In other words, there exist a uniform constant $C>0$ such that we have 
$$ 0 \leq \cG \leq C (\pi^* \omega + i d\t \wedge d\bar\t )$$
on $Y$.
Thanks to the Sobolev embedding theorem, this quasi-plurisubharmonic function $\Phi$ is of class $C^{1,\a}$ for any $\a\in (0,1)$. 

\begin{rem}
\label{rem-pre-001}
We emphasis that a geodesic potential $\Phi$ and an $\ep$-geodesic potential $\Phi_{\ep}$ (see equation (\ref{int-001}))
are always independent of the imaginary part of $\tau$. 
For this reason, we can think of that  $\Phi$ 
and $\Phi_{\ep}$ are periodic functions with period $2\pi$ in the direction $s$.  
Therefore, they are actually $S^1$-invariant 
functions defined on $\cR\times X$,
where $\cR: = [0,1]\times S^1$ is a cylinder \emph{(\cite{C00})}. 
In other words, the defining domain of $\cG$ and $\cG_{\ep}$ can be taken 
as a compact complex K\"ahler manifold with boundary. 

For the same reason, we abuse the complex variable $\tau$ and its real part $t$
from now on, and hope that this will be clear to readers from the context.

\end{rem}

\bigskip

\subsection{The Mabuchi functional }
On the space $\cH_{[\omega]}$, Mabuchi  (\cite{BM}) introduced the following energy functional 
$$\cM: = \underline R \cE - \cE^{Ric\omega} + H, $$
where the constant $\underline R$ is the average of the scalar curvature
$$ \underline R = \frac{n c_1(X)\cdot [\omega]^{n-1}}{[\omega]^n}. $$
The energy functional  $\cE$ is defined for any $\vp\in\cH_{[\omega]}$ as 
$$ \cE (\vp): = \frac{1}{n+1} \sum_{i=0}^n \int_X \varphi \omega^i \wedge \omega_{\vp}^{n-i}.$$
The twisted energy functional $\cE^{\a}$ (by a closed smooth $(1,1)$ form $\a$)
 is defined as
$$ \cE^{\a} (\vp): =  \sum_{i=0}^{n-1} \int_X \vp \omega^{i} \wedge \omega_{\vp}^{n-i-1}\wedge \a.$$
Finally, the entropy functional is 
$$ H(\vp): = \int_X \left( \log\frac{\omega_{\vp}^n}{\omega^n} \right) \omega_{\vp}^n. $$

Consider the pull back $\pi^*\omega$ of the K\"ahler form on $X$, and then it is a  smooth positive $(1,1)$-current on $Y$. 
Suppose $\Phi$ is a $\pi^*\omega$-plurisubharmonic function on $Y$ (see equation (\ref{pre-0})), which corresponds to a geodesic $\cG$.
Then its restriction $$\vp_{\t}: = \Phi|_{X_\t}$$ on the fiber $X_{\t}$ is actually a $\omega$-plurisubharmonic function and 
has the $\cC^{1,\bar 1}$-regularities. It is observed (\cite{BB}) that 
$\cM$ can be defined on such functions. 
Therefore, we can write the Mabuchi functional along the geodesic $\cG$ as 
$$  \cK(\t): = \cM (\vp_\t), \ \ \t\in\Gamma.$$

Next, we  introduce the following modified versions of the Mabuchi functional
due to Berman and Berndtsson (\cite{BB}) .
For any large constant $A >0$,
we define the \emph{A-truncated Mabuchi functional} along the geodesic $\cG$ as 
\begin{equation}
\label{pre-100}
 \cK^{\Psi_A}(\t):= \underline R \cE (\vp_\t) - \cE^{Ric\omega} (\vp_\t) + \int_{X} \log\left(  \max\left\{ \frac{\omega^n_{\vp_\t}}{\omega^n} ,  \frac{h_A}{\omega^n}     \right\}  \right) \omega_{\vp_\t}^n. 
\end{equation}
Here $h_A$ is a continuous metric on the relative canonical line bundle $K_{Y/ \Gamma}$, 
and it is constructed as follows. 

Let $\chi_0$ be a smooth metric on the line bundle $K_{X}$,
and $k_0$ be a positive number such that $dd^c \chi_0 + k_0\omega > 0$.
If set $\chi: = \pi^*\chi_0 - k_0 \Phi$, then we have 
$$ dd^c\chi = \pi^* \chi_0  + k_0 \pi^* \omega - k_0 (\pi^*\omega+ dd^c \Phi) \geq -k_0 \cG.  $$
Taking $h_A: = e^{\chi -A}$,
we obtain on $Y$
$$  dd^c \log h_A \geq - k_0 ( \pi^*\omega + dd^c \Phi), $$
for some positive integer $k_0$.

As explained  in  Remark (\ref{rem-pre-001}), 
The values of these energy functions $\cK, \cK^{\Psi_A}$
do not depend on $\Im\tau$. 
Hence they can be thought of defining 
on the unit interval $[0,1]$. 
Moreover, it is proved (see Theorem (3.4), \cite{BB}) that $\cK^{\Psi_A}$ is a convex function on $[0,1]$,
and it converges to $\cK$ as $A\rightarrow \infty$ by the dominated convergence theorem. 
Eventually, we conclude the following convexity result (\cite{BB}, \cite{CLP}). 

\begin{theorem}
\label{thm-pre-0000}
The Mabuchi functional $\cM$ is convex and continuous along a $\cC^{1,\bar 1}$-geodesic $\cG$.
\end{theorem}

\section{Energies on the $\ep$-geodesics}
Suppose we have two points $\vp_0, \vp_1\in \cH_{[\omega]}$.
For each $\ep>0$ small enough,
there exists a smooth $S^1$-invariant strictly $\pi^*\omega$-plurisubharmonic function $\Phi_{\ep}$ on $\Gamma\times X$ satisfying
\begin{equation}
\label{int-001} 
 \cG_{\ep}^{n+1} = ( \pi^*\omega + dd^c \Phi_{\ep})^{n+1} = \ep \sqrt{-1}dt\wedge d\bar t \wedge \omega^n,
 \end{equation}
with boundary conditions
$$\Phi_{\ep}(0, \cdot) = \vp_0 (\cdot);\ \ \  \Phi_{\ep}(1,\cdot) = \vp_1(\cdot).$$ 
Then we say that $\cG_{\ep}: =  \pi^*\omega + dd^c \Phi_{\ep} $ is the $\ep$-geodesic 
connecting with $\vp_0$ and $\vp_1$.
From the uniform ellipticity of equation (\ref{int-001}), this $\ep$-geodesic is uniquely determined, when its boundary values are fixed. 

It is proved in (\cite{C00}) that $\cG_{\ep}$
is uniformly bounded in the $\cC^{1,\bar 1}$-norm. 
Moreover, we have known that $\Phi_{\ep} \rightarrow \Phi $ in $C^{1,\a}$-norm for any $\a\in (0,1)$,
and also in weakly $W^{2,p}$-norm for all $1 < p < \infty$.

In general, we write the restriction of the geodesic potential on each fiber $X_t: = \{ t\}\times X$ as 
$$ \vp(t, \cdot): = \Phi|_{X_t}; \ \ \  \omega_{\vp}: = \cG|_{X_t} = \omega+ dd_{X}^c \vp. $$
Similarly, we have for the $\ep$-geodesic potential
$$ \vp_{\ep}(t, \cdot): =  \Phi_{\ep}|_{X_t}; \ \ \  \omega_{\ep}: = \cG_{\ep}|_{X_t} = \omega + dd_X^c \vp_{\ep}.$$
Define an operator $\rho$ on the space of all K\"ahler metrics on $Y$ as 
$$  \rho(g): = g_{t\bar t} - g^{\bar\b\a} g_{\a \bar t} g_{t\bar\b}. $$
Writing  $\rho_{\ep}: = \rho (\cG_{\ep})$,
a standard computation shows the following equation 
\begin{equation}
\label{int-002} 
 ( \pi^*\omega + dd^c \Phi_{\ep})^{n+1} = \rho_{\ep} \sqrt{-1}dt\wedge d\bar t \wedge \omega_{\ep}^n.
 \end{equation}
Hence the $\ep$-geodesic equation (\ref{int-001}) can be re-written as 
\begin{equation}
\label{int-003} 
\rho_{\ep}    =  \ep  \frac{\omega^n}{ \omega_{\ep}^n}.
 \end{equation}
 
 \begin{rem}
 \label{rem-int-002}
 It is a well known fact that the fiberwise volume element 
 $\omega^n_{\ep}/\omega^n$ converges to $\omega_{\vp}^n/\omega^n$
 weakly in $L^p$, for all $p>1$, possibly passing to a subsequence. 
 In fact, the $C^{1,\a}$ convergence of $\vp_{\ep}$ to $\vp$ implies 
 the convergence of the Monge-Amp\`ere measure $\omega_{\ep}^n$ to $\omega^n_{\vp}$ in the sense of currents.
 With the uniform upper bound of the volume elements, this further implies that $\omega^n_{\ep}$
 converges to $\omega_{\vp}^n$ in the weak sense of measures (Proposition (2.2), \cite{BT0}).   
 
On the other hand, we can infer from the uniform $L^{\infty}$-norm of the sequence $\omega_{\ep}^n/\omega^n$
that the following holds. There exists a non-negative $L^{\infty}$ function $f_{\infty}$ satisfying 
$$ \frac{\omega_{\ep}^n}{\omega^n} \rightarrow f_{\infty}, $$
in the weakly $L^p$-norm for all $p>1$, possibly after passing to subsequence. 
Therefore, the two measure $\omega_{\vp}^n$ and $f_{\infty}\omega^n$ 
must coincide with each other, and our claim follows. 
 \end{rem}

\subsection{Energy and entropy}
Denote $\cK_{\ep}: \Gamma \rightarrow \mathbb{R}$ by the restriction of the Mabuchi functional $\cM$ to the $\ep$-geodesic as 
$$ \cK_{\ep} (t): = \underline{R}\cE (\vp_{\ep}) - \cE^{Ric\omega} (\vp_{\ep}) + \int_X \log\frac{\omega^n_{\ep}}{\omega^n} \omega^n_{\ep}.$$
Then the $A$-truncated Mabuchi functional along the $\ep$-geodesic can also be introduced as 
$$ \cK_{\ep, A} (t): =  \underline{R}\cE (\vp_{\ep}) - \cE^{Ric\omega} (\vp_{\ep})  + \int_X \log\max \left\{   \frac{\omega_{\ep}^n}{\omega^n}, \frac{h_A}{\omega^n}      \right\} \omega_{\ep}^n,$$
where $  h_A : = h_{\ep, A} = e^{\chi_{\ep} -A} $ is a smooth volume  whose associated curvature is greater than $-C\cG_{\ep}$ for some fixed positive constant $C$.
More precisely, we construct this auxiliary element as before. 

Let $\chi_0$ be a smooth metric on the line bundle $K_{X}$,
and $k_0$ be a positive number such that $dd^c \chi_0 + k_0\omega > 0$.
Then we set $\chi_{\ep}: = \pi^*\chi_0 - k_0 \Phi_{\ep}$, and hence 
$$ dd^c\chi_{\ep} = \pi^* \chi_0  + k_0 \pi^* \omega - k_0 (\pi^*\omega+ dd^c \Phi_{\ep}) \geq -k_0 \cG_{\ep}.  $$
We note that  the sub-index $\ep$ in the notation $h_{A}$ is omitted, since 
$\chi_{\ep}$ is uniformly bounded in its $C^{1, \bar 1}$-norm, and converges uniformly to 
$\chi: = \pi^*\chi_0 - k_0 \Phi $
in $C^{1,\a}$-norm for any $\a\in(0,1)$. 
Therefore, this omission will be harmless for our later argument. 

This new energy function $\cK_{\ep, A}$ is continuous in $t$, since the metric $\omega_{\ep}$
and the volume form $h_A$ both are continuous on $Y$, and moreover the maximum 
$\max\{\omega^n_{\ep}, h_A \}$ is uniformly bounded below away from zero.

We can further compute the complex Hessian of this energy function in the sense of local currents,
and our almost convexity result reads as follows. 

\begin{theorem}[Chen-Li-P\u aun, \cite{CLP}]
\label{int-thm-001}
For each positive number $A$, there is a uniform constant $C_A > 0$ such that the function 
\begin{equation}
\label{int-0091}
 \widetilde{\cK}_{\ep, A}: = \cK_{\ep, A}(t) - \ep C_A t(1-t)
 \end{equation}
is convex and continuous on $[0,1]$ for each $\ep>0$ small. 
\end{theorem}

For the convenience of readers, 
we recall the proof of the above Theorem briefly.  
A similar idea of this proof will be used in Section (4) and (6). 
\begin{proof} [Sketch of the proof]

Suppose $v$ is a locally compact supported smooth test function on $\Gamma$, and then we have 

\begin{eqnarray}
\label{int-004}
\langle dd^c \cK_{\ep, A}, v \rangle &=& \underline{R}\int_{\Gamma\times X} v (\pi^*\omega + dd^c\Phi_{\ep})^{n+1}
\nonumber\\
&-& \int_{\Gamma\times X} v (\pi^*\omega + dd^c\Phi_{\ep})^{n}\wedge \pi^* Ric(\omega)
\nonumber\\
&+& \int_{\Gamma\times X} v  dd^c \left( \max \left\{   \log\frac{\omega_{\ep}^n}{\omega^n}, \log\frac{h_A}{\omega^n} \right\} \right) \wedge (\pi^*\omega + dd^c\Phi_{\ep})^{n}. 
\end{eqnarray}

Taking the fiberwise integral $\int_{X_t}$ as the push forward operator acting on the currents from $\Gamma\times X$ to $\Gamma$,
we obtain the following equation.

\begin{eqnarray}
\label{int-005}
 dd^c \cK_{\ep, A}(t) &=& \frac{\underline{R}}{n+1}\int_{X_t} \cG_{\ep}^{n+1}  - \int_{X_t}  Ric(\omega) \wedge \cG_{\ep}^n
\nonumber\\
&+& \int_{X_t}   dd^c \left( \max \left\{   \log\frac{\omega_{\ep}^n}{\omega^n}, \log\frac{h_A}{\omega^n} \right\} \right)  \wedge \cG_{\ep}^{n}. 
\end{eqnarray}
On the one hand, we have seen locally 
\begin{equation}
\label{int-006} 
\frac{1}{n+1}\left(\Delta_{\cG_{\ep}}  \log{h_A} \right) \cG_{\ep}^{n+1}  =  dd^c  \log{h_A} \wedge\cG_{\ep}^{n} \geq - k_0 \cG_{\ep}^{n+1}.
 \end{equation}
On the other hand,
we denote the function $f_{\ep}: \Gamma\times X \rightarrow \mathbb{R}$ by the equality
$$ \left. \frac{\omega^n_{\ep}}{\omega^n} \right|_{X_t} = e^{f_{\ep} (t, \cdot)},$$
and introduce the following open sub-level set 

\begin{equation}
\label{int-0061}
 P_{\ep, A}: = \left\{ (t,z)\in \Gamma\times X; \ \ \  f_{\ep}(t,z) >  \log \frac{h_{ A}}{\omega^n} (t,z)        \right\}. 
\end{equation}

It is proved in (\cite{CLP}) that there exists a positive number $c_A$ such that the following inequality is satisfied in $P_{\ep, A}$

\begin{equation}
\label{int-007} 
\left( dd^c f_{\ep} - Ric(\omega)  \right) \wedge \cG_{\ep}^n \geq -c_A \cG_{\ep}^{n+1}.
 \end{equation}
We emphasis that $c_A$ only depends on the constant $A$, the background metric $\omega$ and the uniform upper bound of $\omega_{\ep}$, 
and it can be assumed to be increasing in $A$. 
Hence we have the following estimate locally in $P_{\ep, A+1}$

\begin{equation}
\label{int-008} 
\left( \Delta_{\cG_{\ep}} \log\omega_{\ep}^n \right) \cG_{\ep}^{n+1} \geq -c_{A+1} \cG_{\ep}^{n+1}.
 \end{equation}

Now we can utilize a theorem by Greene-Wu (Lemma (5.2), \cite{CLP}), 
and establish the following inequality in an open neighbourhood of each point on $\Gamma\times X$. 
\begin{equation}
\label{int-009} 
\left( \Delta_{\cG_{\ep}} \max\{ \log\omega_{\ep}^n, \log h_A \} \right) \cG_{\ep}^{n+1} \geq -c'_{A} \cG_{\ep}^{n+1},
 \end{equation}
where $c'_A: = \max\{ c_{A+1}, (n+1)k_0 \}$. 
Thus we infer the following inequality holds 
\begin{equation}
\label{int-010} 
\left( dd^c \max\left\{ \log\frac{\omega_{\ep}^n}{\omega^n},  \log\frac{h_A}{\omega^n}        \right\} - Ric \omega \right) \wedge \cG_{\ep}^n \geq - c'_A \cG_{\ep}^{n+1},
 \end{equation}
globally on $\Gamma\times X$.
Combing this estimate with equation (\ref{int-005}), our result follows. 

\end{proof}

One may expect that $\cK_{\ep, A}$ converges to the energy $\cK$. 
If so, then we can directly conclude the convexity of $\cM$ by  Theorem (\ref{int-thm-001}).  

The convergence of the energy parts of $\cM$, i.e. $\cE(\vp_{\ep}) \rightarrow \cE(\vp)$ and $\cE^{\a}(\vp_{\ep}) \rightarrow \cE^{\a}(\vp)$
as $\ep\rightarrow 0$, is indeed true (\cite{BB}, \cite{LL}). 
However, there is a rather severe difficulty: 
the convergence of the fiber-wise volume element 
$$ \frac{\omega_{\ep}^n}{\omega^n} \rightharpoonup \frac{\omega_{\vp}^n}{\omega^n} $$
is only known to be in the weakly $L^p$ sense for any $1<p< \infty$. 
Unfortunately, this is not enough to conclude the convergence of the entropy functional $H(\vp_{\ep})$. 

Nevertheless, we have the following estimate on the entropy functionals. 
Recall that the entropy functional along a geodesic $\cG$ is defined as 
$$  H(\vp): = \int_X \kf(\vp) \log \kf(\vp) \cdot \omega^n, $$
where $\kf(\vp): = \omega_{\vp}^n / \omega^n$,
and its truncated version is
$$  H_{A}(\vp): =  \int_X \kf (\vp) \log \kf_{A} (\vp) \cdot \omega^n, $$
where 
$$ \kf_A (\vp): = \max \left\{   \frac{\omega_{\vp}^n}{\omega^n}, \frac{h_A}{\omega^n}     \right\}. $$ 

As before, we omit the sub-index $\ep$ in the definition of $\kf_A$ as in $h_A$'s. 
Then the energy function $\cK_{\ep, A}$ can be re-written as 
\begin{equation}
\label{int-012}
\cK_{\ep, A} (t) = \underline{R}\cE (\vp_{\ep}) - \cE^{Ric\omega} (\vp_{\ep}) +  H_A(\vp_{\ep}). 
\end{equation}

Finally we state the following lower semi-continuity type property (Lemma (4.8), \cite{CLP}) for the truncated entropy functionals. 
\begin{lemma}
\label{lem-int-001}
There exists a function $\eta: \mathbb{R}^+ \rightarrow \mathbb{R}^+$ 
with $\eta(x) \rightarrow 0$ as $x\rightarrow +\infty$ such that we have 
\begin{equation}
\label{int-013}
 \liminf_{\ep\rightarrow 0}  H_A (\vp_{\ep})   \geq  H_A (\vp) - \eta(A),
 \end{equation}
for all $A$ large enough. 
\end{lemma}
We emphasis that $\eta(A)$ is independent of $\ep$ and $t$.

\subsection{The affine energy}
From now on, we assume that the Mabuchi functional $\cM$ is affine along $\cG$, namely, 
the energy function $\cK(t)$ is linear on $[0,1]$. 
Denote $\cK_{A}$ by the following limit for each $t\in [0,1]$ and $A$ large 
$$ \cK_A(t): = \limsup_{\ep\rightarrow 0} \cK_{\ep, A}(t) = \limsup_{\ep\rightarrow 0} \widetilde{\cK}_{\ep, A} (t),$$
where $\widetilde{\cK}_{\ep, A}$ is defined in equation (\ref{int-0091}).
Then $\cK_A$ is a convex function on the unit interval $[0,1]$,
since the $\limsup$ of a sequence of convex functions which are locally bounded above is still convex. 

Eventually, we will see that this convex function $\cK_{A}$ obtained from taking the limit of $\cK_{\ep, A}$
is exactly equal to $\cK^{\Psi_A}$ (see Corollary (\ref{cor-ae-001})).  
First, we note that $\cK_A$ is actually a decreasing sequence in $A$. 
This is because  $\cK_{\ep, A}$ is a decreasing sequence in $A$ for each $\ep$ fixed.
In fact, we have
$$H_{A'}(\vp_{\ep}) \leq H_A(\vp_\ep), $$ 
 for each $t\in [0,1]$ and  any $A' \geq A$, since $\kf_{A'}(\vp_\ep) \leq \kf_A (\vp_\ep)$ and $\kf(\vp) \geq 0$ in this case. 
Bearing this in mind, we  conclude the following result. 

\begin{theorem}
\label{thm-ae-001}
Suppose the Mabuchi functional $\cM$ is affine along a geodesic $\cG$. 
Then there is a positive number $A_0$ such that we have 
$$ \cK_A (t) = \cK (t), $$
for all $t\in[0,1]$ and $A \geq A_0$. 
\end{theorem}
\begin{proof}
Up to a linear function on $[0,1]$, we can assume that the Mabuchi functional 
$\cM$ is identically zero along $\cG$, namely we have on $[0,1]$
$$\cK(t)\equiv 0. $$ 
The first observation is that we have 
\begin{equation}
\label{ae-000}
 \cK_A(0) = 0; \ \ \  \cK_A (1) = 0, 
 \end{equation}
since the $\ep$-geodesic potential $\Phi_{\ep}$ coincides with the geodesic potential $\Phi$ for each $\ep$ on the boundary of $\cR \times X$.
Then it is easy to see that $\cK_{\ep, A} (0) = \cK_{\ep, A} (1) = 0$ for all $A$ large enough. 

As a convex function on $[0,1]$, $\cK_A$ is upper semi-continuous near the boundaries, and then we have 
$$ \limsup_{t\rightarrow 0, 1} \cK_A (t) \leq 0. $$
Thanks to the convexity again, $\cK_A$ must be below the line segment joining its two boundaries. 
Therefore, it is non-positive under our assumption and we have 
\begin{equation}
\label{ae-001} 
\cK_A(t) \leq 0,
 \end{equation}
for all $t\in[0,1]$. 
On the other hand ,
we have for each $t\in (0,1)$
\begin{equation}
\label{ae-002}
\cK_A (t) \geq \cK(t) - \eta(A), 
\end{equation}
This directly follows from Lemma (\ref{lem-int-001}), where we have proved  
$$ \limsup_{\ep\rightarrow 0} H_A (\vp_{\ep}) \geq H_A (\vp) - \eta(A).$$

Next define a new function on $[0,1]$ as 
$$ \widetilde{\cK}(t): = \limsup_{A \rightarrow +\infty} \cK_A (t). $$
For the same reason, $\widetilde{\cK}$ is a convex function which verifies the boundary condition 
$\widetilde{\cK}(0)= \widetilde{\cK}(1) = 0$ by equation (\ref{ae-000}). 
Moreover, inequality (\ref{ae-002}) implies that 
$\widetilde{\cK}(t) \geq 0$ for each $t\in (0,1) $,
since $\eta(A)\rightarrow 0$ as $A\rightarrow +\infty$ by  Lemma (\ref{lem-int-001}). 
Therefore, we conclude that $\widetilde{\cK}$ must be identically equal to zero for all $t\in [0,1]$.

However, as mentioned before, $\cK_A$ is actually a decreasing sequence in $A$, 
namely, we have $\cK_A(t) \searrow \widetilde{\cK}(t)$ for each $t\in [0,1]$.
Therefore, the following inequality follows 
$$\cK_A(t) \geq 0, $$
for each $t\in [0,1]$.
 Combining this with equation (\ref{ae-001}), we have for all $A$ large enough 
$$\cK_A(t) \equiv 0 $$
on $[0,1]$,  and our result follows. 

\end{proof}

In other words, the limit of the $A$-truncated Mabuchi functional on the $\ep$-geodesics 
will coincide with $\cM$ for all $A$ large enough, provided the linearity of the Mabuchi functional along $\cG$. 

\subsection{Gap phenomenon }

When the Mabuchi functional is affine along $\cG$, we can prove (Lemma (3.1), \cite{LL}) that the $A$-truncated Mabuchi functional $\cK^{\Psi_A}$ (equation (\ref{pre-100}))
also coincide with $\cK$ for all $A$ large. 
This implies  
\begin{equation}
\label{ae-0025}
H_A(\vp) = H(\vp),
\end{equation}
on each fiber $X_t$, for all $A$ large enough. 
Furthermore, we can infer the so called \emph{gap phenomenon} (Proposition (3.2), \cite{LL})
for the fiber-wise volume element of the geodesic $\cG$ from the above equation.
Denote $P$ by the following measurable subset
$$ P: =  \left\{  (t,z)\in \Gamma\times X;\ \ \  \kf (\vp) > 0      \right\}, $$
and $P_A$ by the following sub-level set
$$ P_A:= \left\{  (t,z)\in \Gamma\times X;\ \ \  \kf (\vp) > \frac{e^{\chi-A}}{\omega^n}        \right\}. $$ 
For each $t\in[0,1]$, we introduce 
their fiber-wise restrictions as 
$$ P_t: =       P \bigcap X_t; \ \ \      P_{A,t}: = P_A \bigcap X_t. $$
The gap phenomenon says that 
there is a large constant $A_0 $ such that for all $A\geq A_0$, and each $t\in[0,1]$, 
the following equality holds 
\begin{equation}
\label{ae-0055}
X_t =  P_{A,t} \bigcup P_t^c
\end{equation}
up to a set of measure zero. 
In other words,
there exists a uniform constant $\k_0 > 0$ such that either we have 
on each fiber $X_t$
\begin{equation}
\label{ae-006}
\kf(\vp) > \k_0 \ \ \emph{or}\  \ \ \ \kf (\vp)= 0
\end{equation}
almost everywhere. 
We emphasis that $\k_0$ does not depend on $t$ or $x$.

\bigskip

Thanks to this gap phenomenon, we can prove that the limit of the energy function $\cK_{\ep, A}$ 
exists as $\ep\rightarrow 0$ on each fiber $X_t$ , and it actually coincides with $\cK$
for all $A$ large. 

\begin{corollary}
\label{cor-ae-001}
Suppose the Mabuchi functional $\cM$ is affine along a geodesic $\cG$.
Then for all $A$ large enough, the following limit exists for each $t\in [0,1]$ and satisfies 
\begin{equation}
\label{ae-0026}
 \lim_{\ep\rightarrow 0} \cK_{\ep, A} (t) = \cK^{\Psi_A}(t) =\cK (t).
 \end{equation}
\end{corollary}
\begin{proof}
For each $t\in [0,1]$ fixed, we define 
$$ \widetilde{H}_A (t)  : = \liminf_{\ep\rightarrow 0} H_A(\vp_\ep). $$
Thanks to Lemma (\ref{lem-int-001}), we have 
\begin{equation}
\label{ae-003}
\lim_{A\rightarrow +\infty} \widetilde{H}_A(t) \geq H(\vp),
\end{equation}
by equation (\ref{ae-0025}). 
However, $ \widetilde{H}_A$ is actually a decreasing sequence in $A$.
This follows from its construction and the fact that $H_{A}(\vp_{\ep})$ is decreasing for each $\ep$ fixed. 
Hence, we conclude 
\begin{equation}
\label{ae-004}
 \widetilde{H}_A(t) \geq H(\vp).
\end{equation}
On the other hand, 
Theorem  (\ref{thm-ae-001}) implies for all $A$ large 
\begin{equation}
\label{ae-005}
\limsup_{\ep\rightarrow 0} H_A(\vp_\ep) =  H(\vp) = H_A(\vp).
\end{equation}
Combined the two inequalities above together,  we have 
\begin{equation}
\label{ae-005}
  H(\vp) = \limsup_{\ep\rightarrow 0} H_A(\vp_\ep) \geq  \liminf_{\ep\rightarrow 0} H_A(\vp_\ep) \geq  H (\vp),
\end{equation}
at each $t\in[0,1]$, and then our result  follows.

\end{proof}

\subsection{$L^2$-Convergence}

Recall that the $\ep$-geodesic potential $\Phi_{\ep}$ converges to the geodesic potential $\Phi$
in the weak $C^{1,\bar 1}$-norm. 
Fixing a fiber $X_\t$, 
we can pick up a convergent subsequence of the volume elements on this fiber 
\bigskip
\begin{equation}
f_\l (x) : = \kf(\vp_{\ep_{\l}})(x); \ \ \   f(x): = \kf(\vp)(x),
\end{equation}
satisfying 
$$ f_\l \rightarrow f $$
weakly in $L^p$ for all $p\geq 1$ on $X_\t$.
It is interesting to know whether we have strong convergence in $L^p$ for some $p>1$ for this sequence or not, 
and it turns out that this is indeed the case if the Mabuchi functional is affine.

\begin{theorem}
\label{thm-l2-001}
Suppose the Mabuchi functional $\cM$ is affine along a geodesic $\cG$. 
Then the fiber-wise volume element of the $\ep$-geodesic converges to the 
volume element of the geodesic in the strong $L^2$ sense as $\ep\rightarrow 0$,
possibly after passing to a subsequence. 
More precisely, we have 
$$ \lim_{\l\rightarrow +\infty} || f_\l - f ||_{L^2} = 0,  $$
on each fiber $X_\t$.
\end{theorem}

\begin{proof}
According to Chen-Tian (Theorem (7.1.1), \cite{CT}), 
the $L^2$-convergence of the difference $(f_\l - f)$
follows from the convergence of the entropy  as 
\begin{equation}
\label{l2-018}
\lim_{\l\rightarrow \infty} \int_X \kf(\vp_{\ep_\l}) \log \kf(\vp_{\ep_\l}) = \int_X \kf(\vp) \log \kf (\vp),
\end{equation}
and this is proved in the Corollary (\ref{cor-l2-001}) below. 
\end{proof}

\begin{corollary}
\label{cor-l2-001}
Suppose the Mabuchi functional is affine along a geodesic $\cG$.
Then on each fiber $X_{\t}, \t\in [0,1]$, we have 
$$ \lim_{\ep\rightarrow 0} H(\vp_{\ep}) = H(\vp). $$
\end{corollary}

\begin{proof}
In the proof of Corollary (\ref{cor-ae-001}),
we have seen that the $A$-truncated entropy $H_A(\vp_{\ep})$ is decreasing in $A$ for each $\ep$ fixed. 
Moreover, we proved 
$$ \limsup_{\ep\rightarrow 0} H_A(\vp_\ep) = H_A(\vp) = H(\vp), $$
for all $A$ large enough. 
Therefore, we conclude 
the following inequality
\begin{equation}
\label{l2-019}
\limsup_{\ep\rightarrow 0} H(\vp_\ep) \leq \limsup_{\ep\rightarrow 0} H_A(\vp_\ep) = H(\vp).
\end{equation}

On the other hand, by the lower semi-continuity property of the entropy functional, we have 

\begin{equation}
\label{l2-020}
\liminf_{\ep\rightarrow 0} H(\vp_\ep)  \geq H(\vp),
\end{equation}
and then our result follows.

\end{proof}

\begin{rem}
\label{rem-l2-0011}
Under our assumptions, this Theorem says that for each subsequence $f_{\l}$
converges to $f$ in the weakly $L^2$-norm actually converges in the strong $L^2$-norm.
Moreover, the Corollary (\ref{cor-l2-001}) implies the convergence of the Mabuchi functional along the $\ep$-geodesic,
namely, we have 
$$ \lim_{\ep\rightarrow 0} \cK_{\ep} (t) = \cK(t), $$
for each $t\in [0,1]$.
\end{rem}

\begin{rem}
\label{rem-l2-0022}
As a simple application of Theorem (\ref{thm-l2-001}),
the fiber-wise volume element $f_\l$ of the $\ep$-geodesic converges to 
the volume element $f$ of the geodesic in measure. 
Moreover, thanks to the Riesz-Lebesgue Theorem, we have 
$$f_\l \rightarrow f $$
almost everywhere on each fiber, possibly after passing to a subsequence. 
\end{rem}

\section{The $\ep$-affine energy and non-degneracy }
Suppose the Mabuchi functional $\cM$ is affine along a geodesic $\cG$.
Due to the convergence $\widetilde{\cK}_{\ep, A} \rightarrow \cK_A$ of a sequence of convex functions, the first derivative $\widetilde{\cK}'_{\ep, A}(t)$
also converges uniformly to the slope $k: = \cK'_A(t)$ on the closed interval $[\delta, 1 - \delta ]$ for any $\delta >0$ small. 
However, the boundary behaviour of $\widetilde{\cK}'_{\ep, A}(t)$ is unclear in general. 
Therefore, we impose the following condition on the boundary of $\cG$. 

\begin{defn}
\label{def-ch-001}
The Mabuchi functional $\cM$ is essentially affine along a geodesic $\cG$, if 
$\cK(t)$ is a linear function on $[0,1]$ with slope $k$ and we have 
\begin{equation}
\label{ch-001}
- \int_X \dot{\vp}_t (  R_{\vp_t}  - \underline{R}) \omega_{\vp_t}^n  = k,
\end{equation}
at $t=0, 1$. 
\end{defn}

It is proved in Berman-Berndtsson (\cite{BB}) that the one side inequality of equation (\ref{ch-001}) 
always holds at the two boundaries.
\begin{equation}
\label{bb-001}
  - \int_X \dot{\vp}_0 (  R_{\vp_0}  - \underline{R}) \omega_{\vp_0}^n \leq \frac{d}{dt}|_{t = 0^+} \cK(t) = k,  
\end{equation}
and 
\begin{equation}
\label{bb-002}
 - \int_X \dot{\vp}_1 (  R_{\vp_1}  - \underline{R}) \omega_{\vp_1}^n \geq \frac{d}{dt}|_{t = 1^-} \cK(t) = k.  
\end{equation}

\subsection{A stronger condition}

As we have seen before, the potential $\vp_\ep(t, \cdot)$ of the $\ep$-geodesic converges to 
the potential $\vp(t,\cdot)$ of the geodesic $\cG$ in $C^{1,\a}$-norm for each $\a\in (0,1)$ on $\Gamma\times X$.
Therefore, we have 
\begin{equation}
\label{eaf-001}
-  \lim_{\ep\rightarrow 0} \cK'_{\ep} (t) =  \lim_{\ep\rightarrow 0} \int_X \dot{\vp}_{\ep} (  R_{\vp_\ep}  - \underline{R}) \omega_{\vp_\ep}^n =  \int_X \dot{\vp} (  R_{\vp}  - \underline{R}) \omega_{\vp}^n,
\end{equation}
at the boundaries $t=0,1$, since the metric $\omega_{\ep}$ coincides with $\omega_{\vp}$ at these two boundaries. 
Hence we can compute the first derivative of the energy $\cK_{\ep, A}$ near the boundary as follows. 
\begin{lemma}
\label{lem-eaf-000}
We have 
\begin{equation}
\label{eaf-003}
\lim_{\ep\rightarrow 0}\cK'_{\ep, A} (t) = -\int_X \dot{\vp} (  R_{\vp}  - \underline{R}) \omega_{\vp}^n,
\end{equation}
at $t= 0,1$. 
\end{lemma}

\begin{proof}

We first note that the volume element $\omega_{\ep}^n$ varies smoothly near the boundaries. 
Therefore, there exist a small number $\delta>0$, 
possibly depends on $\ep$ and $A$, such that 
the truncated volume element $\kf_A(\vp_\ep)$ coincides with $\kf(\vp_{\ep})$ for all $A$ large enough 
and all $t\in [0,\delta) \bigcup ( 1-\delta, 1]$.

This implies that the two entropies $H_A(\vp_{\ep})$ and $H(\vp_{\ep})$ also coincides in this small interval. 
Therefore, 
the $A$-truncated Mabuchi functional $\cK_{\ep, A}$ along $\cG_{\ep}$ 
is equal to the Mabuchi functional $\cK_{\ep}$ along $\cG_{\ep}$ 
in a small neighbourhood of the two boundaries, namely, 
 we have 
\begin{equation}
\label{eaf-002}
\cK_{\ep, A} (t) = \cK_{\ep}(t),
\end{equation}
for all $t\in [0,\delta) \bigcup ( 1-\delta, 1]$.
Therefore, we have for all $t$ in this small interval 

\begin{equation}
\label{eaf-003}
\cK'_{\ep, A} (t) = \cK'_{\ep}(t) = -\int_X \dot{\vp}_{\ep} (  R_{\vp_\ep}  - \underline{R}) \omega_{\vp_\ep}^n,
\end{equation}
and then our result follows from equation (\ref{eaf-001}).

\end{proof}

In particular, we conclude the following result. 
\begin{lemma}
\label{lem-eaf-001}
Suppose the Mabuchi functional $\cM$ is essentially affine along a geodesic $\cG$ with slope $k$. 
Then $\cK'_{\ep, A}$ converges uniformly to $k$ on $[0,1]$ for all $A$ large enough. 
\end{lemma}
\begin{proof}
Combing with equation (\ref{eaf-001}) and (\ref{eaf-003}), we have for each $A$ large 

\begin{equation}
\label{eaf-004}
\lim_{\ep\rightarrow 0} \widetilde\cK'_{\ep, A}  (t) =  \lim_{\ep\rightarrow 0} \cK'_{\ep, A}  (t) = k,
\end{equation}
at $t = 0,1$.  
Moreover, the function $\widetilde\cK_{\ep, A}  $ is convex on the interval $[0,1]$,
and then we have the control 
$$  \widetilde\cK_{\ep, A}'  (0) \leq  \widetilde\cK'_{\ep, A}  (t)  \leq  \widetilde\cK'_{\ep, A}  (1),$$
for all $t\in (0,1)$. 
Therefore, $ \widetilde\cK'_{\ep, A}  (t) \rightarrow k$
uniformly as $\ep\rightarrow 0$, and then our result follows from equation (\ref{int-0091}).

\end{proof}

Next we will introduce an even stronger condition  based on the $\ep$-geodesic.

\begin{defn}
\label{def-ch-002}
The Mabuchi functional $\cM$ is $\ep$-affine along a geodesic $\cG$, 
if there exists a uniform constant $ C >0$ such that we have 
\begin{equation}
\label{ll-001}
  \cK_{\ep}'(1) -    \cK_{\ep}'(0)  \leq C \ep, 
\end{equation}
for all $\cK_{\ep}$ with $\ep >0 $ small enough.  

\end{defn}

The geodesic $\cG$ and its $\ep$-geodesic $\cG_{\ep}$ are uniquely determined, when the boundary values $\vp_0, \vp_1$ are fixed. 
Therefore, our definition of $\ep$-affine energy is well posed. 
In fact, this is a stronger condition as follows. 

\begin{lemma}
\label{lem-eaf-0010}
If the Mabuchi functional $\cM$ is $\ep$-affine along a geodesic $\cG$,
then it is essentially affine along $\cG$.
\end{lemma}

\begin{proof}
Thanks to Berman-Berndtsson's estimates (equation (\ref{bb-001}), (\ref{bb-002}))
and the fact that the energy $\cK$ is a convex function, it is enough to prove at the boundary of $\cG$ we have 
$$   \int_X \dot{\vp}_0 (  R_{\vp_0}  - \underline{R}) \omega_{\vp_0}^n  =   \int_X \dot{\vp}_0 (  R_{\vp_1}  - \underline{R}) \omega_{\vp_1}^n.  $$
Due to equation (\ref{eaf-001}), this reduces to prove 
$$ \lim_{\ep\rightarrow 0} \cK'_{\ep}(0) =  \lim_{\ep\rightarrow 0} \cK'_{\ep}(1).   $$ 
One side inequality, namely, 
$$ \lim_{\ep\rightarrow 0} \cK'_{\ep}(1) \leq \lim_{\ep\rightarrow 0} \cK'_{\ep}(0)   $$
is directly implied by our condition (equation (\ref{ll-001})). 
The other side follows from the Lemma (\ref{lem-eaf-000}), and the fact that the energy $\widetilde{\cK}_{\ep, A}$
is a convex function.
\end{proof}

In other words, the first derivative $\cK'_{\ep,A}(t)$ converges to the slope $k$ in a uniform speed controlled by $C\ep$ for all $t\in [0,1]$,
if $\cM$ is $\ep$-affine along $\cG$. This can be inferred from a similar argument as in Lemma (\ref{lem-eaf-001}).
Based on this stronger condition, we state our main theorem on the uniformly non-degenerate of $\cG$ as follows.

\begin{theorem}
\label{thm-ch-001}
Suppose the Mabuchi functional $\cM$ is $\ep$-affine along a geodesic $\cG$.
Then $\cG$ is uniformly fiberwise non-degenerate, 
namely, there exists a uniform constant $\k_1 >0$ such that 
$$   \cG|_{X_t} > \k_1\omega,$$
for almost everywhere $t\in [0,1]$.
\end{theorem}

We emphasis that the constant $\k_1$ does not depend on $t$. 
Before moving on, we need to recall and improve some computations in \cite{CLP}.

\subsection{Estimates}

We are going to find a lower bound of the complex Hessian of the energy $\cK_{\ep, A}$ (equation (\ref{int-001})). 
As explained in the proof of Theorem (\ref{int-thm-001}),
this boils down to evaluate the lower bound of the following $(n+1, n+1)$ form on $P_{\ep, A}$
$$ (dd^c f_{\ep}  - Ric( \omega)  )\wedge \cG_{\ep}^n, $$
Locally near a point $p\in \Gamma\times X$, we write the $\ep$-geodesic as follows
$$ \cG_{\ep}: = g_{t\bar t} dt\wedge d\bar t + \sum_{\a,\b =1}^n \left( g_{t\bar\b} dt\wedge d\bar z^{\b} + g_{\a\bar t} dz^{\a}\wedge d\bar t + g_{\a\bar\b} dz^{\a}\wedge d\bar z^{\b} \right).$$
This is a K\"ahler metric on $\Gamma\times X$,
and its restriction on the fiber $X_t: = \{ t\} \times X$ can be written as 
$$ \omega_{\ep}(t, \cdot): = \cG_{\ep}|_{X_t} = \sum_{\a,\b =1 }^n g_{\a\bar\b} dz^{\a}\wedge d\bar z^{\b}.  $$
Up to a change of holomorphic coordinates on $X_t$, we can assume 
\begin{equation}
\label{com-001}
g_{\a\bar\b} = \delta_{\a\b}; \ \ \ dg =0
\end{equation}
at this particular point $p$. 
Near this point, the $\ep$-geodesic equation can be re-written as 
\begin{equation}
\label{com-002}
\rho_\ep: = g_{t\bar t} - g^{\bar\b\a} g_{\a \bar t} g_{t\bar\b} = \ep e^{-f_{\ep}}.
\end{equation}
Then we introduce another $(1,1)$ form $\chi_{\ep}$ defined by the following equation 
$$ \chi_{\ep}:= \cG_{\ep} -   \rho_{\ep} \sqrt{-1}dt\wedge d\bar t.  $$
This $(1,1)$-form may not be closed anymore, but it is still positive definite on each fiber $X_t$,
and satisfies 
$$ \chi_{\ep}^{n+1} = 0. $$
Hence we have 
\begin{eqnarray}
\label{com-0020}
\cG_{\ep}^n &=& 
\chi_{\ep}^n + n \rho_{\ep} \sqrt{-1} dt \wedge d\bar t \wedge \chi_{\ep}^{n-1}
\nonumber\\
&=& \chi_{\ep}^n + n \rho_{\ep} \sqrt{-1} dt \wedge d\bar t \wedge \cG_{\ep}^{n-1},
\end{eqnarray}
and then we estimate the first factor as follows.

\begin{lemma}
\label{lem-comp-001}
There exists a constant $c_A>0$, possibly depending on $A$, such that we have  
\begin{eqnarray}
\label{comp-002}
&& n\rho_{\ep} (dd^c f_{\ep} - Ric (\omega))\wedge  \sqrt{-1} dt\wedge d\bar t \wedge \cG_{\ep}^{n-1}  
\nonumber\\
&\geq&  \left\{ ( \Delta_{\ep} f_{\ep})  - c_A  \right\} \cG_{\ep}^{n+1},
\end{eqnarray}
in the open set $P_{\ep, A}$.
\end{lemma}
\begin{proof}
The L.H.S. of equation (\ref{comp-002}) can be computed as follows. 
\begin{equation}
\label{com-0021}
 n\rho_{\ep}  dd^c f_{\ep} \wedge \sqrt{-1} dt\wedge d\bar t \wedge \cG_{\ep}^{n-1}  = \ep (\Delta_{\ep} f_{\ep}) \sqrt{-1} dt\wedge d\bar t \wedge \omega^{n},
\end{equation}
and 
\begin{equation}
\label{com-0022}
- n\rho_{\ep}  Ric (\omega) \wedge \sqrt{-1} dt\wedge d\bar t \wedge \cG_{\ep}^{n-1}  = - \ep ( \tr_{\omega_{\ep}} Ric\omega ) \sqrt{-1} dt\wedge d\bar t \wedge \omega^{n}.
\end{equation}
We note that 
there exist a constant $c_A$ satisfying
\begin{equation}
\label{comp-003}
 \tr_{\omega_{\ep}} Ric\omega  <  c_A,
\end{equation}
on $P_{\ep, A}$, since the eigenvalues of $\omega_{\ep}$ are bounded from below (and above) by a uniform constant on this set. 
Therefore, we conclude the inequality 

\begin{eqnarray}
\label{com-0023}
&& n\rho_{\ep} (dd^c f_{\ep} - Ric (\omega))\wedge  \sqrt{-1} dt\wedge d\bar t \wedge \cG_{\ep}^{n-1}  
\nonumber\\
&\geq& ( \Delta_{\ep} f_{\ep})  \cG_{\ep}^{n+1} - c_A \cG_{\ep}^{n+1}.
\end{eqnarray}

\end{proof}

Next, we compute the second factor as follows. 
Introduce the following vector field on $\Gamma\times X$ as 
$$ v: = \frac{\d}{\d t} - g^{\bar\b \a} g_{t\bar\b} \frac{\d}{\d z^{\a}}. $$
Then one observes that this vector field $v$ generates 
the kernel of the $(1,1)$-form $\chi_{\ep}$. 
Hence we have 

\begin{equation}
\label{com-0024}
(dd^c f_{\ep} - Ric (\omega)) \wedge \chi_{\ep}^n = (dd^c f_{\ep} - Ric (\omega)) (v, \bar v)  \sqrt{-1} dt\wedge d\bar t \wedge \chi_{\ep}^n
\end{equation}
Therefore, the goal is to compute the lower bound of the following term
$$ \ddbar \log\det (g_{\a\bar\b}) (v, \bar v), $$
and then we have the following computation.

\begin{lemma}
\label{lem-com-001}
We have 
\begin{eqnarray}
\label{com-015}
 \ddbar \log\det (g_{\a\bar\b}) (v, \bar v) &=&  ||\dbar_X v ||^2_{\omega_{\ep}} + \ep e^{-f_{\ep}} |\nabla_{\ep} f_{\ep}|^2 - \ep e^{-f_{\ep}} ( \Delta_{\ep} f_{\ep} )
 \nonumber\\
 &=& || \dbar v ||^2_{\cG_{\ep}} -  \ep e^{-f_{\ep}} ( \Delta_{\ep} f_{\ep} ).
\end{eqnarray}
\end{lemma}
\begin{proof}
At the chosen point $p$,
a standard computation shows the following equation (here we are using the Einstein summation convention)
\begin{eqnarray}
\label{com-003}
 \ddbar \log\det (g_{\a\bar\b}) (v, \bar v) &=& g_{t\bar t, \a\bar\a} - \sum_{\a,\b} | g_{t\bar\b, \a}|^2 
 \nonumber\\
 &-& g_{t\bar\g} g_{\g \bar t, \a\bar\a}  -  g_{\g\bar t} g_{t\bar\g, \a\bar\a} 
 \nonumber\\
 &+ & R_{\a\bar\g} g_{t\bar\a} g_{\g\bar t}
\end{eqnarray}

Taking the Laplacian $\Delta_{\ep}$ with respect to the metric $\omega_{\ep}$ on the both sides of equation (\ref{com-002}), we have 
\begin{eqnarray}
\label{com-004}
 g_{t\bar t, \a\bar\a} - \ep \Delta_{\ep} (e^{-f_{\ep}}) &=& \sum_{\a, p} |g_{p\bar t, \bar\a}|^2 + \sum_{\b, q} |g_{t\bar q, \bar\b}|^2
 \nonumber\\
 & -& R_{q\bar p} g_{p\bar t} g_{t\bar q} 
 \nonumber\\
& + &   g_{t\bar p} g_{p\bar t, \a\bar\a}  + g_{p\bar t} g_{t \bar p, \a\bar \a}.
\end{eqnarray}
Combing with the two equations above, it turns out that we have 
\begin{equation}
\label{com-005}
 \ddbar \log\det (g_{\a\bar\b}) (v, \bar v) = ||\dbar_X v ||^2_{\ep} + \ep e^{-f_{\ep}} |\nabla_{\ep} f_{\ep}|^2 - \ep e^{-f_{\ep}} ( \Delta_{\ep} f_{\ep} ).
\end{equation}

Furthermore, we can improve the above equality as follows.
First, we observe 
$$  \dbar_X v = - g^{\bar\b \a} g_{t\bar\b,\bar\lambda} d\bar z^{\lambda} \otimes\frac{\d}{\d z^{\a}}. $$
In other words, it can be written in tensors as 
\begin{equation}
\label{com-006}
\dbar_{\lambda} v^{\a} =  -  g_{t\bar\a,\bar\lambda}; \ \ \  \dbar_{\lambda} v^{t}=0,
\end{equation}
at the point $p$. 
Moreover, we have in the time direction
\begin{equation}
\label{com-007}
\dbar_{t} v^{\g} =  -  g_{t\bar t, \bar\g} + g_{t\bar\b}g_{\b\bar\g, \bar t} ; \ \ \  \dbar_{t} v^{t}=0. 
\end{equation}
By differentiating equation (\ref{com-002}) once we get 
\begin{equation}
\label{com-008}
\ep e^{-f_{\ep}} (\d_{\a} f_{\ep}) = g_{t\bar t,\a} - g_{t\bar p} g_{p\bar t, \a}  - g_{p\bar t} g_{t\bar p,\a},
\end{equation}
and similarly 
\begin{equation}
\label{com-009}
\ep e^{-f_{\ep}} (\d_{\bar\b} f_{\ep}) = g_{t\bar t,\bar\b} - g_{t\bar q} g_{q\bar t, \bar\b}  - g_{q\bar t} g_{t\bar q,\bar\b}.
\end{equation}
Hence we have 
\begin{eqnarray}
\label{com-010}
( \ep e^{-f_{\ep}} )^2 |\nabla_{\ep} f_{\ep}|^2 
 &=&g^{\bar\b\a}  \left\{   (  g_{t\bar t,\a} -  g_{p\bar t} g_{\a\bar p,t} ) - g_{t\bar p} g_{p\bar t, \a}   \right\}
 \nonumber\\
&& \cdot  \left\{   (  g_{t\bar t,\bar\b} -  g_{t\bar q} g_{q\bar\b, \bar t}  ) -   g_{q\bar t} g_{t\bar q,\bar\b} \right\}
\nonumber\\
&=& g^{\bar\b\a} ( \d_t v^{\bar\a} - g_{t\bar p} \d_p v^{\bar\a} ) (\dbar_t v^{\b} - g_{q\bar t} \dbar_q v^{\b} )
\nonumber\\
&=& \sum_{\b} |\dbar_t v^{\b} |^2 + g_{t\bar\mu} g_{\lambda\bar t} \d_{\mu} v^{\bar\b} \dbar_{\lambda} v^{\b}
\nonumber\\
&-& g_{\mu\bar t} \d_t v^{\bar\b} \dbar_{\mu} v^{\b} - g_{t\bar\lambda} \dbar_t v^{\b} \d_{\lambda} v^{\bar\b}. 
\end{eqnarray}

On the other side, we can compute the inverse matrix of $\cG_{\ep}$ at the point $p$.
First notice that 
$$\det \cG_{\ep} (p) =  \ep e^{-f_{\ep}(p)},$$
since  $g_{\a\bar\b} = \delta_{\a\b}$ at this point. 
Then a standard calculation shows the following equations.  
$$ \cG_{\ep}^{\bar t t} = (\ep e^{-f_{\ep}})^{-1}; \ \ \  \cG_{\ep}^{\bar t p} = -  (\ep e^{-f_{\ep}})^{-1} g_{t\bar p}, $$
and also 
\begin{equation}
\label{com-011}
\cG_{\ep}^{\bar p q} = \delta_{p q} +    (\ep e^{-f_{\ep}})^{-1} g_{p\bar t} g_{t\bar q} 
\end{equation}
Therefore, the four terms on the RHS of equation (\ref{com-010}) can be re-written as 
\begin{equation}
\label{com-012}
(\ep e^{-f_{\ep}})^{-1} \sum_{\b} |\dbar_t v^{\b} |^2 = \cG_{\ep}^{\bar t t} g_{\a\bar\b}  \dbar_{t} v^{\a} \d_{t} v^{\bar\b};
\end{equation}
\begin{equation}
\label{com-013}
-   (\ep e^{-f_{\ep}})^{-1}   g_{\mu\bar t} \d_t v^{\bar\b} \dbar_{\mu} v^{\b}= \cG_{\ep}^{\bar\mu t} g_{\a\bar\b} \dbar_{\mu} v^{\a} \d_t v^{\bar\b};
\end{equation}
\begin{equation}
\label{com-014}
(\ep e^{-f_{\ep}})^{-1} g_{t\bar\mu} g_{\lambda\bar t} \d_{\mu} v^{\bar\b} \dbar_{\lambda} v^{\b} = \cG_{\ep}^{\bar\lambda\mu} g_{\a\bar\b} \dbar_{\lambda} v^{\a} \d_{\mu} v^{\bar\b}
- g_{\a\bar\b} \dbar_{\lambda} v^{\a} \d_{\lambda} v^{\bar\b}. 
\end{equation}
Combining with equations (\ref{com-010}) - (\ref{com-014}) above, our equation (\ref{com-015}) follows.

\end{proof}

\begin{rem}
This new equality (the second line of equation (\ref{com-015})) is an improvement of what we have obtained in (\cite{CLP}).
It will not be used in the proof of our main theorem. 
However, it reveals the following important fact: the vanishing of the complex hessian $ \ddbar \log\det (g_{\a\bar\b}) (v, \bar v) $
implies the holomorphicity of the vector field $v$, both in $t$ and $X$ directions as $\ep\rightarrow 0$, provided with enough regularities of $\cG$ and $\cG_{\ep}$. 
\end{rem}

\bigskip

\begin{prop}
\label{prop-com-415}
We have the following integral estimate: 
\begin{eqnarray}
\label{com-415}
 && \widetilde\cK'_{\ep, A+1}  (1) -  \widetilde\cK'_{\ep, A+1}  (0) \geq  \int_{P_{\ep, A}} || \dbar v ||^2_{\cG_{\ep}}  i dt\wedge d\bar t\wedge \omega_{\ep}^n
 \nonumber\\
 &\geq&    \int_{P_{\ep, A}}||\dbar_X v ||^2_{\omega_{\ep}}  i dt\wedge d\bar t\wedge \omega_{\ep}^n + \ep  \int_{P_{\ep, A}}  |\nabla_{\ep} f_{\ep}|^2  i dt\wedge d\bar t\wedge \omega^n.
\end{eqnarray}
\end{prop}

\begin{proof}
Combining with Lemma (\ref{lem-comp-001}) and Lemma (\ref{lem-com-001}), 
we can estimate the lower bound of the complex hessian of the energy $\cK_{\ep, A}$ as follows. 

\begin{eqnarray}
\label{com-0150}
&& \left( dd^c \max\left\{   \log\frac{\omega_{\ep}^n }{ \omega^n }, \log\frac{h_A}{\omega^n}      \right\} - Ric\omega \right) \wedge \cG_{\ep}^n + c'_A \cG_{\ep}^{n+1} 
 \nonumber\\
 &\geq&    \chi_{P_{\ep, A-1}} \left(    ||\dbar_X v ||^2_{\omega_{\ep}} + \ep e^{-f_{\ep}} |\nabla_{\ep} f_{\ep}|^2    \right) i dt\wedge d\bar t \wedge \omega_{\ep}^n,
\end{eqnarray}
where $\chi_{P_{\ep, A-1}}$ is the characteristic function of the set $P_{\ep, A-1}$.
In fact, the above inequality (equation \ref{com-0150}) is purely local. 
For points in the open set $P_{\ep, A}$, it directly follows from our previous computations. 
For points in the interior of the complement $( P_{\ep,A})^c $, 
it follows from the equation (\ref{int-006}). 
Finally, for points near the boundary $\d P_{\ep, A}$, the R.H.S of equation (\ref{com-0150})
reduces to zero by our choice of $\chi_{\ep, A-1}$, and then the  inequality is implied by Greene-Wu's theorem again. 

According to equation (\ref{int-005}),
we conclude the following integral estimate
\begin{eqnarray}
\label{com-016}
 && \cK'_{\ep, A+1}  (1) -  \cK'_{\ep, A+1}  (0) + \ep C_{A+1} \geq  \int_{P_{\ep, A}} || \dbar v ||^2_{\cG_{\ep}}  i dt\wedge d\bar t\wedge \omega_{\ep}^n
 \nonumber\\
 &\geq&    \int_{P_{\ep, A}}||\dbar_X v ||^2_{\omega_{\ep}}  i dt\wedge d\bar t\wedge \omega_{\ep}^n + \ep  \int_{P_{\ep, A}}  |\nabla_{\ep} f_{\ep}|^2  i dt\wedge d\bar t\wedge \omega^n.
\end{eqnarray}
Replacing the energy $\cK_{\ep, A}$ by its modified version $\widetilde\cK_{\ep, A}$ 
as in Theorem (\ref{int-thm-001}),
our result follows. 
\end{proof}

Based on this integral estimate, we further obtain the following partial $L^2$ control 
on the gradient of the fiberwise volume element. 

\begin{lemma}
\label{lem-com-002}
Suppose the Mabuchi functional is $\ep$-affine along the geodesic. 
Then there exists a constant $C$ (possibly depending on $t$ and $A$), 
and a subsequence of volume elements $\kf(\vp_{\ep_\l}) $
satisfying 
\begin{equation}
\label{com-020}
 \int_{P_{\ep_\l, A} \bigcap X_{t}} |\nabla \kf (\vp_{\ep_\l}) |^2  \omega^n \leq C,
\end{equation}
for almost everywhere $t\in [0,1]$ and any $\l$ large enough. 

\end{lemma}

\begin{proof}
Suppose the Mabuchi functional is $\ep$-affine along the geodesic, 
and then equation (\ref{eaf-002}) implies 
\begin{eqnarray}
\label{com-017}
 \widetilde\cK'_{\ep, A+1}  (1) -  \widetilde\cK'_{\ep, A+1}  (0)  &= & \cK'_{\ep}(1) - \cK'_{\ep}(0) + 2\ep C_{A+1}
 \nonumber\\
 &\leq & \ep C'_A,
\end{eqnarray}
for some constant $C'_A$.
Therefore, we conclude the following estimate from equations (\ref{com-016}) and (\ref{com-017}).

\begin{eqnarray}
\label{com-018}
\int_{P_{\ep, A}} |\nabla \kf (\vp_{\ep}) |^2 i dt\wedge d\bar t\wedge \omega^n &\leq & C^2 \int_{P_{\ep, A}}  \frac{ |\nabla \kf (\vp_{\ep}) |^2}{\kf^2(\vp_\ep)} i dt\wedge d\bar t\wedge \omega^n
\nonumber\\
& \leq & C^2 \int_{P_{\ep, A}}  |\nabla \log \kf(\vp_\ep) |^2  i dt\wedge d\bar t\wedge \omega^n
\nonumber\\
& \leq & C'  \int_{P_{\ep, A}}  |\nabla_{\ep} \log \kf(\vp_{\ep}) |^2  i dt\wedge d\bar t\wedge \omega^n \leq C''_A,
\end{eqnarray}
where we used the fact $f_{\ep} = \log \kf(\vp_{\ep})$, and $\omega_{\ep}$ is uniformly bounded from the above.  
Moreover, if we take its fiberwise integral as 
$$F_{\ep,A}(t):  = \int_{P_{\ep, A} \bigcap X_{t}} |\nabla \kf (\vp_{\ep}) |^2  \omega^n, $$
then Fatou's lemma implies 

\begin{equation}
\label{com-019}
\int_0^1 \liminf_{\ep} F_{\ep, A} (t) dt \leq \liminf_{\ep} \int_0^1 F_{\ep,A} (t) dt \leq C''_A.
\end{equation}
Therefore, for almost everywhere $t\in [0,1]$, there exist a constant $C_{t}$
and a subsequence $ F_{\ep_\l, A} $ such that 
$$\lim_{\l\rightarrow +\infty} F_{\ep_\l, A} (t)\leq C_{t}. $$
Hence the result follows. 
\end{proof}

\begin{rem}
\label{rem-ch-001}
Let $\vp_{\ep_\l}$ be such a subsequence picked up as in the Lemma (\ref{lem-com-002}). 
we still have uniform $C^{1,\alpha}$-convergence of $\vp_{\ep_\l}$ to the fiberwise geodesic potential $\vp$. 
Then the same argument as in Remark (\ref{rem-int-002}) implies that 
its volume element $\kf(\vp_{\ep_{\l}})$ actually converges to $\kf(\vp)$
in the weakly $L^2$-norm, possibly after passing to a further subsequence. 
Thank to Theorem (\ref{thm-l2-001}), we can conclude 
that $\kf(\vp_{\ep_\l}) \rightarrow \kf(\vp)$ in the strong $L^2$-norm on $X_t$ in the same time. 
\end{rem}

\section{ The uniform Positive lower bound }
To deal with the non-degeneracy of the fiber-wise volume element of $\cG$,
we would like to utilise 
the partial $W^{1,2}$-estimate obtained in equation (\ref{com-020}) and the $L^2$-convergence of the volume elements.
However, the difficulty is that the integral on the LHS of this equation is not taken on the whole manifold $X$,
and the integration domain varies with respect to $\ep$ and $A$. 
In order to overcome this difficulty, we first investigate a local model as follows. 

Suppose $f_{\l}$ is a sequence of positive smooth  functions on the domain $D: = [0,1]^m \subset \mathbb{R}^m$
with uniformly bounded $L^{\infty}$-norm,
and $f$ is an $L^{\infty}$ non-negative function on $D$ such that 
$f_\l$ converges to $f$ in $L^2$-norm.
We further assume that 
the function $f$ satisfies the gap phenomenon, namely, 
there exists a constant $\k_0 >0$ such that we have 
$$ \{ f> \k_0\} \bigcup \{ f =0 \} = D, $$
up to a set of measure zero.
In the following, we set  
$$ P: = \{ f >  \k_0 \}; \ \ \    P^c: = \{ f =0 \}. $$
Let $\k$ denote a positive continuous function on $D$ and we set
$$ P_{\l, \k}: = \{ x\in D; \ \ \  f_\l(x) > \k \}. $$
Then the following result is crucial. 

\begin{prop}
\label{prop-nd-001}
Suppose we have $\mu(P) > 0$ and $\max \k < \k_0 /4$ on $D$,
and assume that there exists a constant $C_1 >0$, such that the following estimate holds 
\begin{equation}
\label{nd-001}
\int_{P_{\l, \k}} |\nabla f_\l |^2  < C_1,
\end{equation}
for a fixed $\k$ and all $\l$ large enough. 
Then $f > \k_0$ almost everywhere on $D$.

\end{prop}

First we will prove that Proposition (\ref{prop-nd-001})
holds in $\mathbb{R}$, namely, we first assume $D = [0,1]$, 
and then the following fact is clear by H\"older's inequality.

\begin{lemma}
\label{lem-nd-001}
Suppose $u$ is a smooth positive function on $[0,1]$, and we assume 
that $ u(a) > 2k $ and $u(b) < k $ for some $0\leq a < b \leq 1$ and a constant $k >0$.  
Then we have 
$$ |b -a | > \frac{k^2}{ \int_a^b ( u'(t) )^2 dt }. $$

\end{lemma}

Suppose $E_1, E_2 $ are two subsets of $D$, 
and we denote $d(E_1, E_2)$ by the distance between them 
$$d(E_1, E_2): = \inf_{x_1\in E_1, x_2\in E_2} d (x_1, x_2).  $$
Then we have  the following observation. 

\begin{lemma}
\label{lem-nd-0010}
Suppose $P, Q$ are two non-empty disjoint subsets of the interval $[0,1]$.
Assume that the union of $P, Q$ is the whole interval $[0,1]$ up to a set with measure zero.
Then we have $d(P, Q)=0$.
\end{lemma}

\begin{proof}
We will prove by contradiction. 
Suppose the distance between the two sets is positive as 
$$  d(P, Q)  > \delta >0. $$ 


\bigskip

Take a large number $m\in \mathbb{Z}$ satisfying $ \frac{1}{m} < \frac{\delta}{100} $.
For each point $p\in P$, we define the following open interval as 
$$ \mathcal{I}_{p,m}: = ( p - \frac{1}{m}, p+ \frac{1}{m}), $$
and similarly for each point $q\in Q$
$$ \mathcal{I}_{q, m} :=  (q- \frac{1}{m}, q+ \frac{1}{m}). $$
We note that $\mathcal{I}_{p,m}$ is disjoint from the set $Q$ for each $p\in P$,
and $\mathcal{I}_{q,m}$ is disjoint from the set $P$ for each $q\in Q$ from our constructions. 
In fact, we have 
$$ \mathcal{I}_{p,m} \bigcap   \mathcal{I}_{q, m} = \emptyset, $$  
for all $p\in P, q\in Q$. 
Therefore, the following two unions 
$$ \mathcal{U}: = \bigcup_{p\in P} \mathcal{I}_{p,m},\ \ \ \mathcal{V}: = \bigcup_{q\in Q} \mathcal{I}_{q,m}$$
are mutually disjoint open subsets of the interval $[0,1]$. 

Moreover, there is a subset $E\subset [0,1]$ with measure zero satisfying 
$$P\bigcup Q \bigcup E = [0,1]. $$
We claim that $E$ is contained in the union $\mathcal{U}\bigcup \mathcal{V}$. 
Otherwise, there is a point $a\in E$ such that we have 
$$ d(\{a\}, P) > \frac{1}{2m}, \ \ \ d(\{a\}, Q) > \frac{1}{2m}. $$
Then the open interval $(a- \frac{1}{2m}, a + \frac{1}{2m})$ must be contained in $E$,
but this is impossible since $E$ has measure zero. 
Therefore, we conclude that the union $\mathcal{U}\bigcup \mathcal{V}$ is exactly the interval $[0,1]$,
which is impossible since $[0,1]$ is a connected set.

\end{proof}

After passing to a subsequence, we can further  assume that $f_\l $ converges to $f$ almost everywhere
on $[0,1]$ due to the $L^2$ convergence (see Remark (\ref{rem-l2-0022})). 
Therefore, there is a subset $E\subset [0,1]$ with measure zero such that 
$f_\l \rightarrow f$ in the poinwise sense outside of $E$. 
Then we are ready to prove the $1$-dimensional case as follows.
 
\begin{lemma}
\label{lem-nd-002}
If we have $n=1$, then Proposition (\ref{prop-nd-001}) holds.  
\end{lemma}
\begin{proof}
We will prove by contradiction. Suppose that $P^c$ has a positive measure. 
Thanks to Lemma (\ref{lem-nd-0010}), we observe that the following two non-empty sets are not separable 
in distance, namely, we have 
$$ d(P- E, P^c - E)  = 0.  $$

Then there exists a sequence of pairs $(a_j, b_j)\in ( P-E ) \times (P^c -E) $
such that $|a_j - b_j| \rightarrow 0 $ as $j\rightarrow \infty$. 
Without loss of generality, we assume $a_j < b_j$ in the following. 
Now fix an arbitrary pair $ (a_j, b_j) $. 
Due to the pointwise convergence of $f_{\l}$ to $f$ outside of $E$, 
we have 
$$f_\l (a_j) > \frac{3\k_0}{4}; \ \ \ \     f_{\l} (b_j) < \frac{\k_0}{4}, $$
for all $\l$ large enough, 
Therefore, there exists a point $c_j$ (possibly depends on $\l$) in the interval $[a_j, b_j]$ such that $f_\l (c_j) = \frac{\k_0}{2}$ by the continuity of $f_{\l}$.
Moreover, the point $c_j$ can be chosen close enough to $a_j$ such that we have  
\begin{equation}
\label{nd-0010}
 f_\l|_{ [ a_j, c_j ]} > \frac{\k_0}{4},
 \end{equation}
since the first derivative $f'_\l$ is bounded.
Hence we have $ [a_j, c_j]\subset P_{\l, \k}$ as $\max\k < \k_0 /4$.  
Then the following inequality holds 
by Lemma (\ref{lem-nd-001}) and equation (\ref{nd-001})
\begin{equation}
\label{nd-002}
|b_j - a_j | > |c_j - a_j | \geq  \frac{\k^2_0}{16C_1}.
\end{equation}

However, this contradicts to the fact that $d(a_j, b_j)\rightarrow 0$ as $j\rightarrow \infty$, 
and then our result follows.

\end{proof}

\begin{proof}
[Proof of Proposition (\ref{prop-nd-001})]
We will use induction on the dimension $m$. 
Lemma (\ref{lem-nd-002}) implies that the result is true for $m=1$, 
and we assume that Proposition (\ref{prop-nd-001}) holds on $\mathbb{R}^{m-1}$ for some integer $m\geq 2$.

Write the coordinate $x\in [0,1]^m$ as 
$ (x_1,\cdots, x_{m-1}, x_m) = (x', y ) $, where $x' = (x_1, \cdots, x_{m-1})$ and $y = x_m$.
Denote $X_y$ the $(m-1)$-dimensional slice by 
$$ X_y: = \{ x \in [0,1]^m; \ \ \ x = (x', y)  \}.$$
Consider the $L^2$ difference of $f_\l$ and $f$ on each slice as 
$$ F_{\l}(y): = \int_{X_y} |f_\l - f|^2 dx' $$
Then the $L^2$-convergence of $f_\l -f$ on $[0,1]^m$ 
implies that $F_{\l}$ converges to zero 
in the $L^1$-norm on $[0,1]$, 
and then it converges in measure. 
By the Riesz-Lebesgue Theorem, it follows that 
there exists a subsequence $F_{\l_k}$ such that
$$ \lim_{\l_k\rightarrow \infty } F_{\l_k} (y) \rightarrow 0, $$
for almost everywhere $y\in [0,1]$.
We emphasis that this subsequence does not depend on $y$. 
Therefore, after re-writing the sub-index, 
we can assume


\begin{equation}
\label{nd-0025}
 || f_\l - f ||_{L^2(X_y)} \rightarrow 0, \ \ \ \ \l \rightarrow \infty
 \end{equation}
for almost everywhere $y\in [0,1]$.
Moreover, 
equation (\ref{nd-001}) can be written as

\begin{equation}
\label{nd-003}
\int_0^1 \left( \int_{X_y\bigcap P_{\l, \k}} |\nabla f_\l |^2 dx'  \right) dy < C_1
\end{equation}
Then by Fatou's lemma (as we have argued in Lemma (\ref{lem-com-002})), 
there exists a constant $C_y > 0 $ and a subsequence $f_\l$,  possibly depending on $y$, satisfying 
  
\begin{equation}
\label{nd-003}
 \int_{X_y\bigcap P_{\l, \k}} |\nabla_{x'} f_\l |^2  \leq \int_{X_y\bigcap P_{\l, \k}} |\nabla f_\l |^2 dx'  < C_y,
\end{equation}
for almost everywhere $y\in [0,1]$ and all $\l$ large.

Furthermore, the gap phenomenon (either we have $f> \k_0$ or $f =0$ almost everywhere )
must be satisfied on $X_y$ for almost everywhere $y \in [0,1]$. 

Up to this point, we have picked up a general fiber $X_y$ and a subsequence $f_{\l}$ (possibly depends on $y$)
such that the new sequence  $f_{\l}|_{X_y}$ satisfies
all conditions in Proposition (\ref{prop-nd-001}),  possibly except $\mu(P\bigcap X_y) > 0$. 
Then for almost everywhere $y\in [0,1]$, from our induction hypothesis, 
we conclude that either one of the following two cases happens: 

\begin{enumerate}

\item[(1)]

$f> \k_0$ 
almost everywhere on $X_y$;

\smallskip

\item[(2)]
$f = 0$ 
almost everywhere on  $X_y$.

\end{enumerate}
\bigskip

In fact, we claim that the above $\mathbf{Case\ (1)}$ occurs for almost everywhere $y\in [0,1]$,
and then our result follows. 
Otherwise,  the following set  
$$ S: = \{ y\in [0,1];\ \ \  f = 0 \ \ \emph{a.e. on $X_y$}       \}, $$
will have positive measure on $[0,1]$. 
Therefore, the following set has positive measure on $[0,1]^m$. 
$$ \mathcal{S}: =   \bigcup_{y \in S}  X_{y}  $$

Now, 
switch the directions and take another slicing as 
$$ x: = ( x_1, x_2, \cdots, x_m)  = ( y', x'' ),$$
where $x'' = (x_2, \cdots, x_m)$. 
Repeating our previous argument on this new slicing, 
we can also conclude that  for almost everywhere $y' \in [0,1]$, 
either $f> \k_0$ or $f = 0$ as an $L^{\infty}$ function on the slice $X_{y'}$.

Next, we consider the following subset on each slice $X_{y'}$ as 
$$ \mathcal{S}_{y'}: = \bigcup_{y \in S}\left( X_{y'} \bigcap  X_y \right). $$
Thanks to Fubini's Theorem, we have 
$$ \mu(\mathcal{S}) = \int_0^1 \mu(\mathcal{S}_{y'}) dy' > 0. $$
However, the measure $\mu(\mathcal{S}_{y'})$ on $[0,1]^{m-1}$
is actually the same for different $y' \in [0,1]$, 
and then itself must be positive too. 

In other words, 
the set $ P^c \bigcap X_{y'} $ has positive measure for each $y' \in [0,1]$. 
Therefore, we conclude that $f = 0$ almost everywhere on $X_{y'}$ for a general point $y' \in [0,1]$.
This contradicts to the fact that $\mu(P) > 0$, and our claim follows.

\end{proof}

Now we are going to prove the main theorem.

\begin{proof}
[Proof of Theorem (\ref{thm-ch-001})]

For almost everywhere $t\in [0,1]$, 
we fix a fiber $X_t$ such that the estimate in Lemma (\ref{lem-com-002}) holds for a sequence
$ \kf(\vp_{\ep_\l})$.
Then we will prove that the volume element $ \kf (\vp) = \omega_{\vp}^n / \omega^n$
is bounded below by the gap $\k_0$ on $X_t$ as an $L^{\infty}$ function. 
Hence the restriction of the geodesic $\cG|_{X_t}$,
as a metric on the fiber, must have a lower bound 
determined by $\k_0$ and its uniform upper bound. 

We will prove by contradiction too. 
Suppose the set $ P^c = \{ \kf(\vp) =0 \}$ has a positive measure on $X_t$.
 As a compact connected complex manifold, the fiber $X_t$ has an open covering by holomorphic coordinate charts, 
namely, we have 
$$ X_t = \bigcup_{ j =1}^N U_j, $$
and the local trivialisation map $\pi_j: B_j \rightarrow U_j$,
where $B_j \subset \mathbb{C}^n $ is an open ball. 
Without loss of generality, we assume that each ball $B_j$ is centred at the origin of $\mathbb{C}^n$
and has radius larger than $2$.
In fact, we can further assume that the manifold $X_t$ is covered by the union of $V_j \subset U_j$, 
where $V_j $ denotes the open set 
$$ V_j: = \pi_j (B),  $$
for the unit ball $B\subset \mathbb{C}^n$. 
Therefore, there exists at least one $j$, such that we have on the coordinate $U_{j}$
\begin{equation}
\label{nd-004}
\mu\left( P^c \bigcap V_{j} \right) >0.  
\end{equation}
On the other hand, there also exists at least one $i$, such that we have on the coordinate $U_{i}$
\begin{equation}
\label{nd-0041}
\mu\left(  P \bigcap V_{i} \right) >0.  
\end{equation}
Denote $\mathcal{J}_1$ by the non-empty collection of all such $j$'s satisfying equation (\ref{nd-004}),
and $\mathcal{J}_2$ the non-empty collection of all such $i$'s satisfying equation (\ref{nd-0041}).
We claim that $\mathcal{J}_1\bigcap \mathcal{J}_2 \neq \emptyset$. 

Suppose not. Then we can take two non-empty open sets as 
$$ \mathcal{U}: =  \bigcup_{j\in \mathcal{J}_1} V_j; \ \ \ \ \mathcal{V}: =  \bigcup_{i\in\mathcal{J}_2}  V_i,$$
and we have $\kf = 0$ pointwise a.e. on $\cU$ and $\kf  > \k_0$ pointwise a.e. on $\cV$.
This implies that 
the intersection of the two open sets is empty. 
Therefore, we have 

\begin{equation}
\label{nd-006}
\cU\bigcap \cV = \emptyset; \ \ \  \cU \bigcup \cV = X_t,
\end{equation}
but this contradicts to the connectness of $X_t$, and our claim follows.

Now pick up such a $j' \in \mathcal{J}_1 \bigcap \mathcal{J}_2$, and then we choose
a subsequence $f_{\l}: = \kf(\vp_{\ep_\l})$ of the volume element as in Lemma (\ref{lem-com-002}) and Remark (\ref{rem-ch-001}).  
Take their restriction to the following domain as 
$$ f_\l: = \kf (\vp_{\ep_\l})|_{D}, \ \ \   f : = \kf(\vp)|_{D},  $$
where $D$ is the $m^{th}$-cube $[-1,1]^m $ contained in $U_{j'}$, and it follows $V_{j'}\subset D$. 

Recall that $P_{\ep, A}$ is the subset of $\Gamma\times X$,  where $e^{f_{\ep}}$ is larger than the function $\frac{e^{\chi -A }}{\omega^n}$,
or equivalently, 
$$ P_{\ep, A}\bigcap X_t = \left\{  x\in X_t;\ \ \   \kf(\vp_{\ep} (t, x)) > \k       \right\}, $$
where $\k = e^{\chi - A} / \omega^n $ is the auxiliary function (equation (\ref{kappa})).
When the constant $A$ is large enough, we can assume $\max_{X_t}\k < \k_0 /4$ for all $t\in [0,1]$. 
Then we put 
$$ P_{\l, \kappa}: =  P_{\ep_{\l}, A}\bigcap X_t, $$
and then the $L^2$-gradient estimate as in equation (\ref{nd-001}) holds for such $f_{\l}$ on $P_{\l, \k}$,
based on the result in Lemma (\ref{lem-com-002}). 

Up to this stage, we observe that 
all the conditions in Proposition (\ref{prop-nd-001})
are satisfied on this domain $D$ for the sequence $f_\l$ with the function $f$. 
Therefore, we conclude that $f > \k_0$ almost everywhere on $D$.
However, this contradicts to the fact that $\{ f =0  \}$ has positive measure on $V_{j'} \subset D$,
and our main result follows.

\end{proof}

\section{Applications}
As one application of the main result (Theorem (\ref{thm-ch-001})), 
we will prove that the Mabuchi functional $\cK_{\ep}: = \cM(\vp_{\ep})$ along the $\ep$-geodesic 
converges to the Mabuchi functional $\cM(\vp)$ along the geodesic, 
not only pointwise on each fiber (Corollary (\ref{cor-ae-001})), but also in its complex Hessian.

For the first step, we will take a closer look at equation (\ref{com-002})  as follows.
Recall that the $\ep$-geodesic $\cG_{\ep}$ is uniformly $\cC^{1,\bar 1}$ on $\Gamma\times X$.
Therefore, we have 

\begin{equation}
\label{app-001}
C_2 >  g_{t\bar t} - g^{\bar\b\a} g_{\a\bar t} g_{t\bar\b} = \frac{\ep\omega^n}{ \omega_{\ep}^n },
\end{equation}
for some uniform constant $C_2$. 
For the same reason, 
the eigenvalues of $\omega_{\ep}$ are bounded from above by a uniform constant $C_0 >0$. 
Moreover, at the normal coordinate of a point $p\in\Gamma\times X$, we can write 
$$ \omega = \sum_{ j=1}^n dz^j \wedge d\bar z^j; \ \ \     \omega_{\ep} = \sum_{j=1}^n \lambda_j  dz^j \wedge d\bar z^j. $$
where $0 < \lambda_1 \leq \cdots \leq \lambda_n$ 
are the $n$ eigenvalues of the metric $\omega_{\ep}$ at this point. 
Hence we have the following inequality. 
\begin{eqnarray}
\label{app-002}
 \tr_{\omega_{\ep}} \omega & =& \sum_{k=1}^n \frac{1}{\lambda_i} 
\nonumber\\
&\leq& \frac{n \lambda_2\cdots \lambda_n}{\lambda_1 \cdots \lambda_n}
\nonumber\\
&\leq & \frac{n C_0^{n-1}}{\det(g_{\a\bar\b})} \leq \frac{nC_2 C_0^{n-1}}{\ep}.
\end{eqnarray}
Then we can obtain a better estimate on the R.H.S. of (\ref{com-0022}) as follows.
\begin{eqnarray}
\label{app-003}
&&  \ep \int_{ \{Ric\omega>0\}} (  \tr_{\omega_{\ep}}Ric\omega ) \omega^n 
 \nonumber\\
 &\leq&  \ep C \int_{ \{Ric\omega>0\}}   ( \tr_{\omega_{\ep} }\omega )\omega^n
 \nonumber\\
& =&  \ep C \int_{P_{\ep, A_1} \bigcap  \{Ric\omega>0\} } ( \tr_{\omega_{\ep} }\omega )\omega^n  + \ep C \int_{(P_{\ep, A_1})^c\bigcap \{Ric\omega>0\} } ( \tr_{\omega_{\ep} }\omega )\omega^n 
 \nonumber\\
& \leq &  \ep C \int_{P_{\ep, A_1} } ( \tr_{\omega_{\ep} }\omega )\omega^n  + C_4 \mu \left( \{ \kf(\vp_{\ep}) \leq \k_1 \} \right)
\nonumber\\
&\leq & \ep C_{A_1} + C_4 \mu \left( \{ \kf(\vp_{\ep}) \leq \k_1 \} \right), 
\end{eqnarray}
where we have used equation (\ref{app-002}) on the third line of the above equation,
and $A_1$ is a fixed constant large enough such that 
$$\max_{X_t} \k_1 < \k_0/4. $$
\begin{rem}
\label{rem-app-001}
We claim that the R.H.S. of equation (\ref{app-003}) converges to zero as $\ep\rightarrow 0$,
after passing to a subsequence.  
In fact, the first factor converges to zero since $A_1$ is fixed.
Moreover, Theorem (\ref{thm-ch-001}) implies that the measure $ \mu( \{ \kf(\vp_{\ep}) \leq \k_1 \} )$
converges to zero, possibly after passing to a subsequence.
Therefore, the second factor converges to zero too. 
\end{rem}

Therefore, we conclude the following result.

\begin{theorem}
\label{thm-app-001}
Suppose the Mabuchi functional is $\ep$-affine along a geodesic. 
Then for almost everywhere $t\in[0,1]$, we have on the fiber $X_t$
$$ \limsup_{\ep\rightarrow 0} \cK''_{\ep}(t) \geq 0.  $$ 
In particular, we have the following convergence of the $L^2$-norms 

\begin{equation}
\label{app-005}
\int_{X_t} || \dbar_X v ||_{\omega_{\ep}}^2 \omega_{\ep}^n, \int_{X_t} || \dbar v ||_{\cG_{\ep}}^2 \omega_{\ep}^n \rightarrow 0,
\end{equation}
as $\ep\rightarrow 0$ for this subsequence. 
\end{theorem}
\begin{proof}

Recall that the lower bound of $dd^c \cK_{\ep,}$ can be estimated 
by computing  the factor 
$$ ( dd^c f_{\ep} - Ric\omega)  \wedge \cG_{\ep}^{n}. $$
Repeat the steps  in Lemma (\ref{lem-comp-001}),
but we stop at equation (\ref{com-0022}) this time. 
Together with Lemma (\ref{lem-com-001}), we obtain as before

\begin{eqnarray}
\label{app-0041}
&& \left( dd^c  \log\frac{\omega_{\ep}^n }{ \omega^n } - Ric\omega \right) \wedge \cG_{\ep}^n 
 \nonumber\\
 &\geq&     \left(    ||\dbar_X v ||^2_{\omega_{\ep}} + \ep e^{-f_{\ep}} |\nabla_{\ep} f_{\ep}|^2    \right) i dt\wedge d\bar t \wedge \omega_{\ep}^n
 \nonumber\\
 & - &  \chi_{ \{ Ric\omega >0 \} } ( \ep \tr_{\omega_{\ep}} Ric\omega ) \sqrt{-1} dt\wedge d\bar t \wedge \omega^{n}.
\end{eqnarray}

Next we perform the integration on an arbitrary interval $[ t_1, t_2]\subset[0,1]$. 
Then the new estimate (equation (\ref{app-003})) shows the following inequality

\begin{eqnarray}
\label{app-016}
&& \cK'_{\ep}  (t_2) -  \cK'_{\ep}  (t_1) 
\nonumber\\
& \geq&  - (t_2- t_1) \left( \ep C_{A_1} + C_4 \mu \left( \{ \kf(\vp_{\ep}) \leq \k_1 \} \right) \right). 
\end{eqnarray}
Therefore, we further obtain 
$$ \cK''_{\ep}(t) \geq   - \ep C_{A_1} - C_4 \mu \left( \{ \kf(\vp_{\ep}) \leq \k_1 \} \right). $$
As in Remark (\ref{rem-app-001}),
the R.H.S. of the above equation converges to zero, possibly after passing to subsequence. 
Then our first result follows. 

Finally, our $L^2$-gradient estimates (equation (\ref{app-005})) also follows from 
an integral estimate based on equation (\ref{app-0041}) and the $\ep$-affine condition on $\cK_{\ep}$ (compare with Proposition (\ref{prop-com-415})). 



\end{proof}

By utilizing the estimate in the above Theorem, i.e. $\dbar v_{\ep} \rightarrow 0$ in the $L^2$ sense, one may expect that 
$v_{\ep}$ would converge to a vector field $v_{\infty}$, and this limit $v_{\infty}$ is holomorphic on $\Gamma\times X$. 
However, this is still unclear to us since the fiberwise volume element $\omega_{\ep}^n$ may not converge uniformly to the 
volume element $\omega^n_{\vp}$ of the geodesic $\cG$. 
Up to this stage, we can only conclude the holomorphicity of $v_{\infty}$ under some special cases. 

\subsection{Special cases}
Suppose the two boundaries $\vp_0, \vp_1$ of $\cG$ are both non-degenerate energy minimizers of $\cM$. 
Then $\cM(\vp_t)$ keeps to be a constant along this geodesic, and it satisfies the $\ep$-affine condition automatically.
Therefore, our Theorem (\ref{thm-ch-001}) implies that the geodesic  $\cG$ is fiberwise uniformly non-degenerate, 
and then the regularities of $\cG$ can be improved by the work of He-Zeng(\cite{HZ}). 
Then we recover one of our result in \cite{LL}.

\begin{theorem}[L.]
\label{thm-app-002}
Suppose the two boundaries of a $C^{1,\bar 1}$-geodesic $\cG$ are both non-degenerate energy minimizers of $\cM$.
Then the geodesic $\cG$ is generated by a holomorphic vector field. 
\end{theorem}

On the other hand, we can assume that the K\"ahler manifold $X$ satisfies $c_1(X)< 0$ or $c_1(X) = 0$. 
Then our computation in Section (4) would recover Chen's estimates in \cite{C00}, under the $\ep$-affine condition.
In conclusion, we can infer the following result by a similar argument as in Section (6), \cite{C00}. 

\begin{theorem}
\label{thm-app-003}
Suppose the manifold $X$ satisfies $c_1(X)< 0$ or $c_1(X) = 0$.
Assume that the Mabuchi functional $\cM$ is $\ep$-affine along a $C^{1,\bar 1}$-geodesic $\cG$.
Then $\cG$ is generated by a holomorphic vector field. 
\end{theorem}

Finally, we would like to emphasis that the boundaries of the geodesic $\cG$ are assumed to be smooth and non-degenerate in our set up. 
Therefore, one possible way to utilize the $L^2$-convergence of $\dbar v_{\ep}$ is to consider their behavior close enough to the boundary. 
Hence we will end up with the following observation, which may be useful in our later consideration. 

\begin{prop}
\label{prop-app-001}
Suppose $\cG$ is a $\cC^{1,\bar 1}$-geodesic connecting two K\"ahler potentials $\vp_0, \vp_1\in \cH$. 
Then its fiberwise volume element has the following convergence near the boundaries
$$ \kf(\vp_t) \rightarrow \kf(\vp_0); \ \ \ \kf(\vp_s) \rightarrow \kf(\vp_1), $$
as $t\rightarrow 0$ and $s\rightarrow 1$  in the $L^2$-sense. 
\end{prop}
\begin{proof}
By a Theorem proved in Chen-Tian (Theorem 7.1.1, \cite{CT}), 
it is enough to show 
$$ \lim_{t\rightarrow 0} \int_X \kf(\vp_t) \log \kf(\vp_t) \rightarrow   \int_X \kf(\vp_0) \log \kf(\vp_0), $$
and 
$$ \lim_{s\rightarrow 1} \int_X \kf(\vp_s) \log \kf(\vp_s) \rightarrow   \int_X \kf(\vp_1) \log \kf(\vp_1). $$
In other words, the entropy $H(\vp_t) (H(\vp_s))$ converges to $H(\vp_0) (H(\vp_1))$ as $t\rightarrow 0$ and $s\rightarrow 1$.
This is true because the Mabuchi functional $\cM $ is convex and continuous up to the boundaries of $\cG$.
\end{proof}


\section{Appendix}
In the following, we will provide a different proof of Theorem (\ref{thm-l2-001}).
This one is more complicated, but it will 
give an accurate estimate for the $L^2$-norm of the difference $f_\l - f$
directly from the convergence of the truncated entropies. 
We expect that this estimate will be useful for some independent interests.

For the beginning, denote $\k$ by the auxiliary function on the fiber $X_\t$ 
\begin{equation}
\label{kappa}
 \k(x, A): = \frac{e^{\chi -A}}{\omega^n} (x).
 \end{equation}
 As before, we omit the sub-index $\ep_{\l}$ in $\chi$ due to the uniform control on their $\cC^{1,\bar 1}$-norms.
Then the following result holds.  

\begin{lemma}
\label{lem-l2-0010}
Suppose we have 
$$ \lim_{\l\rightarrow \infty}  \int_X \kf(\vp_{\ep_\l}) \log\kf_A (\vp_{\ep_\l}) = \int_X \kf(\vp)\log\kf_A(\vp).$$
Then there exists a uniform constant $C$, only depending on the upper bound of $f_\l$ and $f$, satisfying 
\begin{equation}
\label{l2-0001}
\lim_{\l\rightarrow \infty} \int_X (f_\l - f)^2 \omega^n \leq \left(   2+ \frac{8C}{\k_0} \right) \max_{X_{\t}} \k,
\end{equation} 
for all $\k$ small enough (or $A$ large enough). 
\end{lemma}

Here $\k_0$ is the gap of the volume element $\kf(\vp)$ (equation (\ref{ae-006})), which is a fixed constant. 
Therefore, Theorem (\ref{thm-l2-001}) directly follows from Lemma (\ref{lem-l2-0010}) if we take $A\rightarrow \infty$.

\subsection{The maximum function}
In order to prove this lemma, 
the first step is to investigate the following maximum function. 
Define a function $F: [0, +\infty) \rightarrow \mathbb{R}$ as 
$$F(x): = x\log x, $$
and $F(0) = 0$. Then $F$ is a convex continuous function on its domain.
In fact, it is smooth in $\mathbb{R}^+$, and 
 its first and second derivatives are 
$$ F'(x) = \log x +1;\ \ \    F''(x) = x^{-1} $$

Moreover, we can truncate $F$ by a small number $\k>0$ and introduce the following maximum function 
$$ h_{\k}(x): = \max\{ x\log x, x\log\k \}.  $$
This is also a convex and continuous function on $[0, +\infty)$,
and it is piecewise smooth in this domain.
Its first derivative exists everywhere on $\mathbb{R}^+$ except at the point $x=\k$, and we have 
\begin{equation}
\label{l2-001}
h'_{\k}(x) = \left\{ \begin{array}{rcl}
\log\k, 
& \mbox{for} &
 x< \k \\ 
1+ \log x,
& \mbox{for} & x > \k. \\
\end{array}\right.
\end{equation}
Furthermore, we can also compute its second derivative on $\mathbb{R}^+$ as 
\begin{equation}
\label{l2-002}
h''_{\k}(x) = \left\{ \begin{array}{rcl}
0, 
& \mbox{for} &
 x< \k \\ 
\delta(\k),
&\mbox{for}& x=\k \\
\frac{1}{x},
& \mbox{for} & x > \k, \\
\end{array}\right.
\end{equation}
where $\delta(\k)$ is the \emph{Dirac-delta} function at the point $x=\k$. 
We note that the Fundamental Theorem of Calculus is still satisfied for $h_{\k}, h'_{\k}, h''_{\k}$
on the interval $[0,1]$, since we can take the differentiation in the sense of the generalized derivatives.

Fixing a point $x$ on the fiber, we introduce another variable $t\in[0,1]$ and take 
$$u_t:  = t f_\l + (1-t) f = at +b,  $$
where $a: = (f_\l - f)(x)$ and $b: = f(x)$.
We note that $u_t$ is non-negative and has a uniform upper bound for all $t$, $\l$ and $x$,
and it is strictly positive for all $t>0$, or $t=0$ and $f(x) >0$. 
Define a new composition function as 
$$ F_{\l, \k}(t): = h_{\k} (u_t). $$
Observe that this function $F_{\l. \k}$ is actually linear in $t$ for all $u_t$ small. 
Moreover, its derivatives can be written as 
\begin{equation}
\label{l2-003}
F'_{\l, \k}(t) = \left\{ \begin{array}{rcl}
a \log \k, 
& \mbox{for} &
 at +b < \k \\ 
a\log(at +b) +a ,
&\mbox{for}& at+b > \k; \\
\end{array}\right.
\end{equation}

\begin{equation}
\label{l2-004}
F''_{\l, \k}(t) = \left\{ \begin{array}{rcl}
0, 
& \mbox{for} &
 at +b < \k \\ 
a \delta(t_0),
&\mbox{for}& at +b =\k \\
\frac{a^2}{at+b},
& \mbox{for} & at+b > \k, \\
\end{array}\right.
\end{equation}
where $t_0$ is determined by the equation
$$at_0 + b = \k. $$
In particular, we have $F'_{\l, \k}(0) = a\log\k$ if $f(x) < \k$,
and  $F'_{\l, \k}(0) = a\log b + a$ if $f(x) > \k$.  
Then the following convergence holds. 

\begin{lemma}
\label{lem-l2-001}
For all $A$ large enough, we have 
$$ \lim_{\l\rightarrow + \infty} \int_X F'_{\l, \k} (0) \omega^n  = 0.$$
\end{lemma}
\begin{proof}

When the constant $A$ is large enough, we can assume $\k < \k_0/2$ on the fiber $X_{\t}$, 
where $\k_0$ is the gap of the fiber-wise volume element of $\cG$ defined in equation (\ref{ae-006}).
Then the fiber can be completely decomposed into two parts as in equation (\ref{ae-0055}) 
$$ X_{\t} =  P_{A,\t} \bigcup P_{\t}^c, $$
up to a set of measure zero. 
Recall that  the two sets can be re-written as follows 
$$ P_{A, \t} = \{ x\in X_{\t}; \ \ \   f(x) > \k(x)   \}; \ \ \   P^c_{\t}: = \{ x\in X_\t; \ \ \  f(x) =0 \}. $$
Then we have 

\begin{eqnarray}
\label{l2-005}
\int_X F'_{\l, \k} (0) &= &\int_{P_{A, \t}}  F'_{\l, \k} (0)+ \int_{P_{\t}^c}  F'_{\l, \k} (0)
\nonumber\\
&=& \int_{P_{A, \t}} (f_{\l} - f) \log f  + \int_{P_{A, \t}} (f_{\l} - f)  + \int_{P_\t^c} f_{\l}\log\k
\nonumber\\
&= & \int_{P_{A, \t}} (f_{\l} - f) \log f  + \int_{P_{\t}^c} (f_\l  -f)\log\k + \int_{P_{\t}^c} f \log\k
\nonumber\\
&+& \int_X (f_\l - f) - \int_{P^c_\t} (f_\l - f) 
\nonumber\\
&=& \int_X (f_\l - f) \log\max\{  f, \k\}  + \int_X (f_\l  - f) - \int_{P^c_{\t}} f_\l.
\end{eqnarray}
The three terms on the RHS of equation (\ref{l2-005}) will all converge to zero as $\l\rightarrow +\infty$,
since $|| f ||_{L^{\infty}}, || f_\l ||_{L^{\infty}}$ are uniformly bounded and $f_\l \rightarrow f$ weakly in $L^p$ for any $p\geq 1$,
and then our result follows.

\end{proof}

\subsection{The four cases}

Next we will apply the Fundamental Theorem of Calculus on the function $F_{\l, \k}(t)$ and its first derivative,
namely, we have 

\begin{eqnarray}
\label{l2-006}
F_{\l, \k}(1) - F_{\l, \k}(0) &=& h_{\k}(f_\l) - h_{\k}(f)
\nonumber\\
&=& \int_0^{t_0} F'_{\l, \k}(t)dt + \int_{t_0}^1 F'_{\l, \k}(t)dt.
\end{eqnarray}

A first observation is that the point $t_0$ may not be in the integration domain above. 
Suppose the point $x$ is in the subset $P_{\t}^c$,
and then we have $b =f(x) = 0$, $a = f_{\l}(x) > 0$.
Therefore, the point $t_0$ belongs to the interval $(0,1)$ 
if and only if $ 0< \k < a$. 

Otherwise, we have $t_0 \geq 1$ if $ f_\l (x) \leq \k $, but $t_0 \leq 0$ is not possible in this case 
since it means  $\k\leq 0$. 

On the other hand, suppose the point $x$ is in the subset $P_{A,\t}$. 
Then we have $b = f(x) > \k$ and $a = f_\l (x)- f(x)$.
Hence $t_0 \in (0,1)$ if and only if 
$a< 0$ and $b  > \k  > a+b = f_{\l}(x)$.

Otherwise,  when $a\geq 0$, we have $f_l(x) \geq f(x) > \k $ for $t_0 \leq 0$;
or when $a<0$, we have $\k  \leq f_\l (x)$ for $t_0 \geq 1$. 

In conclusion, we distinguish all situations into the following four cases: 
\begin{enumerate}

\item[(i)]
$x\in P_{A,\t}$, $f_\l (x) \geq f(x) > \k $ or $ f(x) > f_\l(x) \geq \k $;
\smallskip\\ 
\item[(ii)]
$x\in P^c_{\t}$,  $f(x) = 0$ and $f_\l (x) \leq \k$;
\smallskip\\ 
\item[(iii)]
$x\in P^c_{\t}$, $f(x)=0$ and $f_\l (x) > \k$;
\smallskip\\ 
\item[(iv)]
$x\in P_{A, \t}$, $f_\l (x) < \k < f(x)$.

\end{enumerate}

We note that these four cases are disjoint from each other. 
Then we will discuss case by case. 
For $\textbf{Case (i)}$, we note that 
$u_t > \k$ and then $h_{\k}(u_t) = u_t \log u_t$ for all $t\in [0,1]$.
Therefore, we can further compute as follows. 

\begin{eqnarray}
\label{l2-007}
F_{\l, \k}(1) - F_{\l, \k}(0) 
&=& F'_{\l,\k}(0) + \int_0^1 \int_0^t F_{\l,\k}''(s) ds dt
\nonumber\\
&=& F'_{\l,\k}(0) + \int_0^1 \int_0^t \frac{a^2 ds}{as +b} dt
\nonumber\\
&\geq& F'_{\l,\k}(0) + \frac{a^2}{2C}, 
\end{eqnarray}
where the constant $C$ is the uniform upper bound of $u_t$. 

For $\mathbf{Case (ii)}$, we note $u_t\leq \k$ and then $h_{\k}(u_t) = u_t \log \k$.
Hence we have 

\begin{eqnarray}
\label{l2-008}
F_{\l, \k}(1) - F_{\l, \k}(0) 
&=& F'_{\l,\k}(0) + \int_0^1 \int_0^t F_{\l,\k}''(s) ds dt
\nonumber\\
&=& F'_{\l,\k}(0). 
\end{eqnarray}

The two cases above are the easy ones. 
For the remaining cases, we will utilise equations (\ref{l2-003}) and (\ref{l2-004}) in the computation. 
In $\mathbf{Case (iii)}$, we note that $h_{\k}(u_t) = u_t \log\k $ for $t\leq t_0$ and $h_{\k}(u_t) = u_t \log u_t$ for $t> t_0$. 
Recall that $t_0 = \k /a$ in this case, and hence the computation follows.

\begin{eqnarray}
\label{l2-009}
F_{\l, \k}(1) - F_{\l, \k}(0) 
&=& \int_0^{\k/a} F'_{\l, \k} (t) dt  + \int_{\k/a}^1 F'_{\l,\k}(t) dt
\nonumber\\ 
&=& \frac{\k}{a} F'_{\l, \k} (0) + \int_0^{\k/a} \int_0^t F''_{\l, \t} (s) ds dt + \int_{\k/a}^1 F_{\l,\k}'(t) dt
\nonumber\\
&=& F'_{\l, \k} (0) + \int_{\k/a}^1 \left\{ a + \int_{\k/a}^t F''_{\l, \k}(s) ds  \right\} dt 
\nonumber\\
&=&  F'_{\l, \k} (0) + (a - \k) + \int_{\k/a}^1 \int_{\k/a}^t \frac{a^2 ds}{as +b } dt
\nonumber\\
&\geq &  F'_{\l, \k} (0) + (a - \k) + \frac{a^2}{2C} (1- \k/a)^2 
\nonumber\\
&\geq &  F'_{\l, \k} (0) + \frac{a^2}{2C} + a(1- \k/C) - \k. 
\end{eqnarray}
Recall that $a > 0$ in this case, and then we have 

\begin{equation}
\label{l2-010}
F_{\l, \k}(1) - F_{\l, \k}(0)  \geq  F'_{\l, \k} (0) + \frac{a^2}{2C}  - \k,
\end{equation}
for all $\k$ small enough. 
Finally, the most difficult one is $\mathbf{Case (iv)}$. 
As before, we first note that $h_{\k}(u_t) = u_t\log u_t$ for 
$t\leq t_0$ and $h_{\k}(u_t)$ and $h_{\k} (u_t) = u_t\log\k$ 
for $t > t_0$. Then we compute in a similar way.

\begin{eqnarray}
\label{l2-011}
F_{\l, \k}(1) - F_{\l, \k}(0) 
&=& \int_0^{\frac{\k - b}{a}} F'_{\l, \k} (t) dt  + \int_{\frac{\k-b}{a}}^1 F'_{\l,\k}(t) dt
\nonumber\\ 
&=& \frac{\k-b}{a} F'_{\l, \k} (0) + \int_0^{\frac{\k-b}{a}} \int_0^t F''_{\l, \t} (s) ds dt + \int_{\frac{\k-b}{a}}^1 F_{\l, \k}'(t) dt
\nonumber\\
&=& F'_{\l, \k} (0) + \int_0^{\frac{\k-b}{a}} \int_0^t \frac{a^2 ds}{as +b} dt + \int^1_{\frac{\k-b}{a}} \left\{     \int_0^{\frac{\k-b}{a}}  F''_{\l,\k}(s)ds  + a \right\}
\nonumber\\
&\geq & F'_{\l, \k} (0)  + a +b   + \frac{(\k-b)^2}{2C}  + \frac{ (\k -b)(a+b -\k)}{C}      -\k.
\end{eqnarray}

Recall that we have $a+b = f_\l(x) >0$, $\k-b = \k - f(x) < 0$ and $a+b - \k = f_\l(x) - \k <0$.
Hence the following estimate holds.

\begin{eqnarray}
\label{l2-012}
F_{\l, \k}(1) - F_{\l, \k}(0)  &\geq&  F'_{\l, \k} (0) + \frac{(\k -b)^2 }{2C} - \k
\nonumber\\
&\geq &  F'_{\l, \k} (0) + \frac{ \k_0^2 }{8C} - \k.
\end{eqnarray}
The last inequality in equation (\ref{l2-012}) holds is because that 
we have picked up $\k < \k_0 /2$, and $P_{A, \t}$ is actually the set where $f(x) > 0$ on the fiber which is equal to 
$$  \{ x\in X_{\t}; \ \ \  f (x) > \k_0   \}$$
up to a set of measure zero by the gap phenomenon.

Combining with equations (\ref{l2-007}) - (\ref{l2-012}) above, we conclude the following inequality after taking the integral on $X_{\t}$. 

\begin{eqnarray}
\label{l2-013}
 \int_X ( h_{\k} (f_\l) - h_{\k} (f)  ) \omega^n &\geq& \int_X F'_{\l, \k}(0) \omega^n -  \max_{X_{\t}} \k  
 \nonumber\\
 &+&  \frac{1}{2C}\int_{\left(  P^c_{\t} \bigcap \{ f_\l \leq \k \}   \right)^c  \bigcap   \left(  P_{A,\t} \bigcap \{    f_l < \k < f     \}         \right)^c             } (f_\l - f)^2 \omega^n
 \nonumber\\
 &+& \frac{1}{8C} \k_0^2 \mu\left(   P_{A, \t} \bigcap \{   f_l < \k < f \}    \right).
\end{eqnarray}
Then we are ready to prove the main theorem in this section.

\begin{proof}
[Proof of Lemma (\ref{lem-l2-0010})]
By our choice on the function $\k$ and the maximum function $h_{\k}$,
it follows  

\begin{equation}
\label{l2-014}
H_A (\vp_{\ep_\l}) - H_A (\vp) =  \int_X ( h_{\k} (f_\l) - h_{\k} (f)  ) \omega^n.
\end{equation}
Thanks to Corollary (\ref{cor-ae-001}), the LHS of equation (\ref{l2-013}) converges to zero as $\l \rightarrow +\infty$.
Meanwhile, our Lemma (\ref{lem-l2-001}) implies the first term on the RHS of equation (\ref{l2-013}) also converges to zero. 
Therefore, it implies  

\begin{equation}
\label{l2-015}
 \frac{8C}{\k_0^2}  ( \max_{X_\t} \k ) \geq \lim_{\l\rightarrow +\infty}  \mu\left(   P_{A, \t} \bigcap \{   f_l < \k < f \}    \right).
\end{equation}

Furthermore, we note that the two subsets $P_{A,\t} \bigcap \{    f_l < \k < f     \}  $ and $P^c_{\t} \bigcap \{ f_\l \leq \k \} $ are mutually disjoint. 
Then the third term on the RHS of equation (\ref{l2-013}) can be decomposed into the following three parts. 
$$ \int_X (f_\l -f)^2  - \int_{P^c_{\t} \bigcap \{ f_\l \leq \k \}  }(f_\l -f)^2 - \int_{P_{A,\t} \bigcap \{    f_l < \k < f     \} } (f_\l -f)^2. $$
The first negative term in the equation above can be estimated as

\begin{equation}
\label{l2-016}
 \int_{P^c_{\t} \bigcap \{ f_\l \leq \k \}  }(f_\l -f)^2 \leq \max_{X_\t} \k^2,
\end{equation}
and the second negative term can be estimated by equation (\ref{l2-015}) as 

\begin{equation}
\label{l2-017}
 \int_{P_{A,\t} \bigcap \{    f_l < \k < f     \} } (f_\l -f)^2 \leq  C^2  \mu\left(   P_{A, \t} \bigcap \{   f_l < \k < f \}    \right).
\end{equation}
Combing with equations (\ref{l2-014}) - (\ref{l2-017}) and take the limit in $\l$, we eventually conclude the following estimate

\begin{equation}
\label{l2-017}
\left(  2+ \frac{8C^3}{\k_0}       \right) \max_{X_\t} \k   \geq \lim_{\l \rightarrow +\infty} \int_X (f_\l - f)^2\omega^n,
\end{equation}
and then our result follows.

\end{proof}

\end{document}